\documentclass{article}
\pdfoutput=1
\usepackage[utf8]{inputenc}
\usepackage{amsthm}
\usepackage{amsmath}
\usepackage{amsfonts}
\usepackage{amssymb}

\usepackage{graphicx}
\usepackage{color}
\usepackage{todonotes}
\usepackage{hyperref}

\theoremstyle{plain}
\newtheorem{theorem}{Theorem}
\newtheorem{proposition}[theorem]{Proposition}
\newtheorem{lemma}[theorem]{Lemma}
\newtheorem{corollary}[theorem]{Corollary}

\theoremstyle{definition}

\newcommand*{\bT}{\mathbb{T}}
\newcommand*{\bR}{\mathbf{R}}
\newcommand*{\bC}{\mathbf{C}}
\newcommand*{\bN}{\mathbf{N}}

\newcommand*{\cR}{\mathcal{R}}

\DeclareMathOperator{\capacity}{\mathrm{cap}}

\DeclareMathOperator{\sign}{\mathrm{sign}}

\DeclareMathOperator{\interior}{\mathrm{Int}}

\begin{document}

\title{Higher Markov and Bernstein inequalities and fast decreasing polynomials with prescribed zeros}
\author{Sergei Kalmykov and  B\'ela Nagy}

\maketitle

\begin{abstract}
Higher order Bernstein- and Markov-type  inequalities are established for trigonometric polynomials on 
compact subsets of the real line 
and algebraic polynomials on compact subsets of 
the unit circle. 
In the case of Markov-type inequalities we assume that 
the compact set satisfies an interval condition. 


Classification (MSC 2010): 41A17, 30C10 
\end{abstract}

\section{Introduction}
Two of the most classical polynomial inequalities are the  Bernstein inequality
(see \cite{BorweinErdelyi}, p.~233 Theorem 5.1.7 or \cite{MMR}, p.~532, Theorem 1.2.5)
\begin{equation*}
\left|P'_n(x)\right|
\le 
\frac{n}{\sqrt{1-x^2}}\|P_n\|_{[-1,1]}, 
\quad x\in (-1,1),
\end{equation*}
and the Markov inequality (see
\cite{BorweinErdelyi}, p.~233 Theorem 5.1.8 
or \cite{MMR}, p.~529 Theorem 1.2.1)
\begin{equation*}
\|P'_n\|_{[-1,1]}\le n^2 \|P_n\|_{[-1,1]},
\end{equation*}
where $P_n$ is an algebraic polynomial of degree of at most $n$, and $\|\cdot\|_{X}$ denotes the sup-norm over the set $X$. 
For a trigonometric polynomial $T_n$ of the degree at most $n$ the following Bernstein-type inequality holds 
(established by M.~Riesz, see \cite{MMR}, p.~532 Theorem 1.2.4 or 
\cite{BorweinErdelyi}, p.~232 Theorem 5.1.4)
\begin{equation*}
\|T'_n\|_{[0,2\pi]}
\le n \|T_n\|_{[0,2\pi]}.
\end{equation*}

There is also an analogue of this inequality for trigonometric polynomials on an interval less than the period 
see \cite{BorweinErdelyi} p.~243.
In 2001, Totik 
developed 
the method of polynomial inverse images to prove an asymptotically 
sharp Bernstein- and Markov-type inequalities for algebraic polynomials on several intervals \cite{Totik2001}, 
and in \cite{IMRN} asymptotically sharp  inequalities were also obtained for
trigonometric polynomials  on several intervals and for algebraic polynomials on several circular arcs on the complex plane. 
The case of one circular arc was 
considered 
earlier
in \cite{NT2013}.
In recently published paper \cite{CAOT} algebraic polynomials on
sets satisfying \eqref{intervalcondition} were considered,
for trigonometric polynomials, see \cite{MarkovInterval}.
The next step in generalization of these result was done in \cite{TZ}, 
where 
asymptotic higher order Markov-type inequalities for algebraic polynomials on compact sets satisfying (\ref{intervalcondition}) 
were established. 

The purpose of the present paper is to extend these results to trigonometric polynomials and 
to algebraic polynomials on subsets of the unit circle
and to present a new type of fast decreasing polynomials. 
Briefly, the approach of Totik-Zhou \cite{TZ}
was to establish the Markov-type inequality for T-sets, 
then for general sets and 
use Fa\`a di Bruno's formula and 
Remez inequality near interior critical points.
The difference here is that we developed 
fast decreasing polynomials with prescribed zeros to deal with interior critical points. 
Moreover, we also establish Bernstein-type inequality.

Sharp higher order Markov-type inequality is established for sets
satisfying the interval condition \eqref{intervalcondition}.
At interior points sharp 
Bernstein-type inequality is also derived which involves much 
slower 
growth order 
($O(n^{2k})$ at endpoints vs. $O(n^k)$ at interior points where $k$-th   derivatives are considered).

The structure of the paper is the following. 
First, notation is introduced, 
and some known, basic results about T-sets are  mentioned.
Then the important density results (for T-sets and regular sets) are recalled.
New results are in Section  \ref{sec:newresults}.
A construction of fast decreasing polynomials with prescribed zeros can also be found here. 
A preliminary, "rough" Markov- and Bernstein-type inequalities are needed for special sets.
Then asymptotically sharp Markov-type inequality is formulated for higher 
derivatives of trigonometric polynomials 
and for algebraic polynomials on subsets of the unit circle. 
Finally, asymptotically sharp Bernstein-type inequalities are established in the trigonometric case as well as in the algebraic case.

\section{Notation, background}
We denote by 
$\bR$ 
the real line, by 
$\bC$ 
the complex plane, 
by $\overline{\bC}$ 
the extended complex plane, 
and by 
$\mathbb{T}$ the unit circle 
and by $\mathbf{N}$ the nonnegative integers. 

We use Fa\`a di Bruno's formula  (or Arbogast's formula; see \cite{KrantzParks}, p. 17 or \cite{Riordan}, pp. 35-37
or \cite{FdBJohnson}): 
if $f$ and $g$ are $k$ times differentiable functions, then
\begin{equation}
\label{faadibruno}
\frac{d^k}{dx^k} f(g(x)) =
\sum \frac{k!}{m_1!m_2!\ldots m_k!}
f^{(m_1+m_2+\ldots+m_k)}(g(x))
\prod_{j=1}^k \left(
\frac{g^{(j)}(x)}{j!}
\right)^{m_j}
\end{equation}
where the summation is for all nonnegative integers $m_1,m_2,\ldots,m_k$ such that
\begin{equation*}
1 m_1 +2 m_2 +\ldots + k m_k=k.
\end{equation*}

\medskip

Let $E\subset[-\pi,\pi)$ be a set which is closed in $[-\pi,\pi)$.
Since we do not consider $E=[-\pi,\pi)$ 
(it is classical), 
we may assume that $E\subset (-\pi,\pi)$.
We consider the corresponding set on the unit circle
\begin{equation*}
E_{\mathbb{T}} := \left\{ \exp(it): \  t\in E\right\}.
\end{equation*}

We use the interval condition:
a compact set $E\subset (-\pi,\pi)$ satisfies the interval condition at $a\in E$ 
if there is a $\rho>0$ such that
\begin{equation}
[a-2\rho,a]\subset E \mbox{ and } 
(a,a+2\rho)\cap E = \emptyset.
\label{intervalcondition}
\end{equation}

We use potential theory, for a background, we refer to \cite{Ransford} or \cite{SaffTotik}.
For a compact set $K\subset \bC$, its capacity is denoted by $\capacity (K)$.
If $\capacity (K)>0$, then the equilibrium measure is denoted by $\nu_K$.
It is known that 
if $K\subset \bR$ is a compact set 
$\nu_K$ is absolutely continuous with respect to Lebesgue measure at interior points of $K$ 
and its density is denoted by $\omega_K(t)$.
It is 
also 
known that if $E\subset (-\pi,\pi)$ 
satisfies the interval condition at a  point $a\in E$, then $\sqrt{|t-a|} \omega_E(t)$ has a finite, positive limit as $t\rightarrow a$.
Similarly, we say that the compact set 
$K\subset \bT$ 
satisfies the interval condition at $e^{ia}$ 
where $a\in (-\pi,\pi)$
if $K=E_\bT$ and for some $E$, $E$ satisfies the interval condition at $a$.
Furthermore, if $K$ satisfies the interval condition at $e^{ia}$ ($a\in(-\pi,\pi)$), then $\sqrt{|e^{it}-e^{ia}|} \omega_K(e^{it})$
has a finite, positive limit as $t\rightarrow a$ too.
Hence we introduce
\begin{gather*}
\Omega(E,a):=\lim_{t\rightarrow a} \sqrt{|t-a|}\,\omega_E(t),
\\
\Omega(K,e^{ia}):=\lim_{t\rightarrow a} \sqrt{|e^{it}-e^{ia}|}\,\omega_K(e^{it}).
\end{gather*}
It is worth noting that $\Omega(.,.)$ is monotone with respect to the set, that is,
if $E_1\subset E_2\subset [-\pi,\pi)$, and both satisfy the interval condition at $a$, 
then $\Omega(E_2,a)\le \Omega(E_1,a)$.
Similar assertion holds for the unit circle.

In the finitely many arcs case,
there is a very useful representation of the density of the equilibrium measure 
(see \cite{PeherstorferSteinbauer}, Lemma 4.1 and also formula (5.11)):
let $K=\cup_{j=1}^m \{\exp(it): a_{2j-1}\le t \le a_{2j}\}$ where
$-\pi<a_1<a_2<\ldots<a_{2m-1}<a_{2m}<\pi$ and put $a_{2m+1}:=2\pi+a_1$.
Then there exist $\tau_j\in (a_{2j},a_{2j+1})$, $j=1,\ldots,m$
such that
\begin{equation}
\label{taucondition}
\int_{a_{2j}}^{a_{2j+1}}
\frac{\prod_{j=1}^m (e^{it}-e^{i\tau_j})
}{
\sqrt{\prod_{j=1}^m (e^{it}-e^{i a_{2j-1}})(e^{it}-e^{i a_{2j}})}
} dt=0 
\end{equation}
where, to be definite,  the branch of the square root 
is chosen so that 
$\sqrt{z}\rightarrow\infty$ as $z\in\bR$, $z\rightarrow +\infty$. 
Actually it should hold that 
\[
(-1)^m i \prod_j e^{i\tau_j}=\sqrt{\prod_j e^{i(a_{2j-1}+a_{2j})}}
\]
but actually the other branch would be just as fine, 
since the right hand side in \eqref{taucondition} is $0$.
Then 
\begin{equation*}
\omega(K,e^{it})=\frac{1}{2\pi}
\frac{\prod_{j=1}^m \left|e^{it}-e^{i\tau_j}\right|
}{
\sqrt{\prod_{j=1}^m \left|e^{it}-e^{i a_{2j-1}}\right| \left| e^{it}-e^{i a_{2j}}\right| }
},
\quad t\in\interior K
\end{equation*}
see \cite{PeherstorferSteinbauer}, formula (5.11).
In this case,
\begin{equation*}
\Omega\left(K, 
e^{i a_k}\right)
=
\frac{1}{2\pi}
\frac{\prod_{j=1}^m \left|e^{i a_k}-e^{i\tau_j}\right|
}{
\sqrt{\prod_{j=1,\ldots,2m, j\ne k} \left|e^{i a_k}-e^{i a_{j}}\right|  }
}.
\end{equation*}

\subsection{Density results}

\label{sec:density}

We use special sets on $(-\pi,\pi)$.
A set $E\subset(-\pi,\pi)$ is called T-set, if
\begin{equation}
E=\{ t\in (-\pi,\pi): |U_N(t)|\le 1\}
\label{Tsetdef}
\end{equation}
for some (real) trigonometric polynomial $U_N$ with degree $N$ which
attains $+1$ and $-1$ $2N$-times.
For a background on T-sets, we refer to Section 3 in \cite{IMRN}.

We define
\begin{equation*}
M(E,a_j)=
M_{a_j}:= 
\frac{\prod_{l=1}^{m} 
\left|e^{i a_j}-e^{i \tau_l}\right|^2
}{
\prod_{l=1,\ldots,2m,l\ne j }
\left|e^{i a_j}-e^{i a_l} \right|
}
\end{equation*}
and obviously, 
\begin{equation*}
M(E,a_j)= 4\pi^2 \Omega^2(E_\bT,e^{i a_j}).
\end{equation*}

\smallskip

Now we recall some monotonicity and continuity results regarding $\Omega(E,a)$ and $M(E,a)$.

For any $\varepsilon>0$, by Lemma~3.4 from \cite{IMRN}  (see p. 3001) 
we can choose 
an admissible polynomial $U_N$  
such that the inverse image set 
$E'=(U_N^{-1}[-1,1])\cap [-\pi,\pi]=\cup_{j=1}^m [a'_{2j-1},a'_{2j}]$  consists of $m$ intervals and it lies close to $E$, that is
$|a'_j-a_j|<\varepsilon$ for all $j=1,\ldots,2m$ and $E'\subset E$.
Also we may assume that $a\in E'$.
Again $j_0$ is such that 
$a\in[a'_{2j_0 -1},a'_{2j_0}]$ and actually $a=a'_{2j_0}$. 
 For numbers $\tau_i$ in (\ref{taucondition})  
it is clear that they are  $C^{1}$-functions of the endpoints $a_j$. 
Then with $M'_a:=M(E',a)$,
we have $\lim_{\varepsilon\rightarrow 0}M'_a= M_a$.
By the monotonicity of $\Omega(.,.)$ in the first variable, we immediately have
that $M_a\le M'_a$.

In other words, for any $\varepsilon>0$, there 
exists 
a T-set 
$E'\subset E$, $a\in E'$ such that $\Omega^2(E'_\bT,e^{ia})\le (1+\varepsilon)\Omega^2(E_\bT,e^{ia})$.

\medskip

Consider an arbitrary compact set $E\subset (-\pi,\pi)$ satisfying the 
interval condition \eqref{intervalcondition}, 
and assume that $E$ is not a union of finitely many intervals. 
The set $[-\pi,\pi]\setminus E$ consists of finitely or countably many intervals open in
$[-\pi,\pi]$: 
\begin{equation*}
[-\pi,\pi]\setminus E
=\bigcup_{j=0}^{\infty} I_j
\end{equation*}
To be definite, we assume that $I_0$ contains $(a,a+2\rho)$. 
Further, for $m\ge 0$ we consider the set 
\begin{gather*}
E_m^+=[-\pi,\pi] \setminus \left(\bigcup_{j=0}^{m}
I_j
\right)=
\bigcup_{j=1}^{m'}[a_{j,m'},b_{j.m'}],
\\
a_{1,m'}\le b_{1,m'}<a_{2,m'}\le b_{2,m'}<\dots<a_{m',m'}\le b_{m',m'}=b_{0,m'}
\notag
\end{gather*}
where $m'=m+1$ (note here, by our assumption $E\subset(-\pi,\pi)$).

Obviously, $E_m^+$ contains $E$ and satisfies the interval condition (\ref{intervalcondition}). 
If $a_{j,m'}=b_{j,m'}$ for some $j$, then we replace this degenerated interval by the interval 
\[
[a_{j,m'}-\lambda_m,a_{j,m'}+\lambda_m]\bigcap [-\pi,\pi],
\]
where $\lambda_m<1/m$ is chosen to be so small that the interval condition (\ref{intervalcondition}) is still satisfied. For the set obtained this way we preserve the notation $E_m^{+}$. 

\smallskip

We also use the famous result of Ancona (see \cite{Ancona}).
If $K\subset \bT$ is any compact set,
$\capacity(K)>0$, then for any $\varepsilon>0$ there exists
$K_1\subset K$ compact set which is
regular for the Dirichlet problem and
$\capacity(K)\le \capacity(K_1)+\varepsilon$.
Furhermore, it is easy to see that
if $K$ satisfies the interval condition 
\eqref{intervalcondition},
then $K_1$ can be chosen such that
it satisfies \eqref{intervalcondition} too.
Let $E_m^-$ be the set coming from Ancona's theorem applied to $E_\bT$ with $\varepsilon=1/m$ and also satisfying the interval condition \eqref{intervalcondition}.

\begin{lemma}\label{l:Omegalimit}
For the two sets $E_m^+$ and $E_m^-$
introduced above,
we have  $\Omega\left( (E_m^{\pm})_\bT, e^{ia}\right) \rightarrow \Omega({E}_{\mathbb{T}}, e^{ia})$ 
holds true
as $m\rightarrow \infty$.
\end{lemma}
For a proof, see e.g. \cite{CAOT}, p. 1295, Proposition 2.3.

\bigskip

\section{New results}
\label{sec:newresults}

We need fast deceasing polynomials with prescribed zeros
and  rough Markov- and 
Bernstein-type inequalities.

\subsection{Fast decreasing trigonometric 
and algebraic polynomials with prescribed zeros}

\label{sec:fastdecrtrig}

Special 
fast decreasing polynomials  with prescribed zeros
are constructed in this subsection.
First, their existence 
are 
established on the real line, then 
in the
trigonometric case.

We tried to find 
this type of 
fast decreasing polynomials 
in the existing literature (e.g. in \cite{LevinLubinsky}, \cite{IvanovTotik},
\cite{Totik1994}, \cite{TotikChristoffel},\cite{TotikSzego},\cite{TotikVarga}, \cite{KuijlaarsVanAssche} 
and Lemma 4.5 on p. 3012 in \cite{IMRN}),
but we did not find the following two results.
Further, possible applications may include estimates for Christoffel functions, etc.

\begin{theorem}
\label{prop:fdpz}
Let $a_0<a_1<\ldots< a_{l_0}<a'<a<x_0<b<b'<a_{l_0+1}<\ldots <a_l<a_{l+1}$ be fixed 
and $k_0,k_1,\ldots,k_l$ be positive integers.
Put $Z(x):=\prod_{j=1}^l (x-a_j)^{k_j}$.
Then there exists $\delta_1>0$ such that for all large $m$
there exists a polynomial $Q(x)$
with degree at most $m$ such that
\begin{gather}
Q(x_0)=1, \label{fdpz:atxnull}
\\
Q^{(j)}(x_0)=0,\quad
j=1,\ldots,k_0,
\label{fdpz:datxnull}
\\
|Q(x)|<1 \mbox{ if } x\in [a_0,a_{l+1}], x\ne x_0, \label{fdpz:peaking}
\\
\left|Q(x)-1\right|\le \exp(-\delta_1 m) 
\mbox{ for } x\in[a,b],
\label{fdpz:high}
\\
\left|Q(x)\right|\le 
\min\left(1,|Z(x)|\right)
\exp(-\delta_1 m) 
\mbox{ for } x\in[a_0,a']\cup[b',a_{l+1}],
\label{fdpz:low}
\\
Q(x) \mbox{ is strictly monotone on } [a',a]\mbox{ and on }[b,b'],
\label{fdpz:change}
\\
Q^{(k)}(a_j)=0, \quad j=1,\ldots,l, k=0,1,\ldots, k_j,
\label{fdpz:zeros}
\\
Q(x)\ge 0 \mbox{ for } x\in[a_0,a_{l+1}].
\label{fdpz:nonneg}
\end{gather}
\end{theorem}
\begin{proof}
In this proof several new pieces of notation are introduced which are used here only and constants are not
redefined from line to line in this proof just for sake of convenience.

Consider $S$, which will be a polynomial satisfying all but one properties, in the form
\begin{equation}
S(x)= C_1 \int_{a_1}^x Z_1(t)\, P_1(t)\, R(t) (t-x_0)^{k_0'}\, dt
\label{Qrepr}
\end{equation}
where
\begin{gather*}
Z_1(t):=\prod_{j=1}^{l} (t-a_j)^{k_j'},
\quad
R(\underline{\tau}; t)=R(t):= \prod_{\substack{j=1\\ j\ne l_0}}^{l-1}(t-\tau_j)
\\
P_0(t)=P_0(\delta,\mu;t):=\left(1-\left(\frac{x-\delta}{c_2}\right)^2\right)^\mu \\
P_1(\alpha,\beta,\lambda,\mu;t)=
(1-\lambda)
P_0(\alpha,\mu;t)
+
\lambda
P_0(\beta,\mu;t)
\end{gather*}
and where $k_0'=k_0$ if $k_0$ is odd
and $k_0'=k_0+1$ if $k_0$ is even,
and for $j=1,\ldots,l$, 
$k_j'=k_j$ if $k_j$ is even and
$k_j'=k_j+1$ if $k_j$ is odd,
and $\tau_j\in [a_j,a_{j+1}]$, $j=1,\ldots,l-1$, $j\ne l_0$,
and $a'<\alpha<a<b<\beta<b'$, 
$\alpha:=(a+a')/2$, $\beta:=(b+b')/2$ and
$\mu$ is large positive integer and 
$c_2:=a_{l+1}-a_0$,
$\lambda\in[0,1]$.
If some of the parameters are fixed or unimportant in the current consideration, then we leave them out, e.g.
$P_0(t)=P_0(\delta,\mu;t)$ and
$P_1(t)=P_1(\mu;t)=P_1(\lambda,\mu;t)=P_1(\alpha,\beta,\lambda,\mu;t)$.

\medskip

The key observation is that
if $S(a_j)=0$ for some $j$, then we immediately have that $S^{(k)}(a_j)=0$, $k=0,1,2,\ldots,k_j$.

\medskip

Some obvious properties immediately follow from the definitions:
$Z_1(t)\ge 0$ (this is why we increased the "multiplicities"),
$P_0(t),P_1(t) \ge 0$ too, $\max\limits_{a_0\le t \le a_{l+1}} P_1(t)\ge 1/2$.
Furthermore, the degree of $R$ is 
$l-2$ 
and $R$ has the same sign over $(a',b')$.
For simplicity, denote $\underline{\tau}_1:=(\tau_1,\ldots,\tau_{l_0-1})$, $\underline{\tau}_2:=(\tau_{l_0+1},\ldots,\tau_l)$ 
and (slightly abusing the notation)
$\underline{\tau}:=(\underline{\tau}_1,\underline{\tau}_2)=(\tau_1,\ldots,\tau_{l_0-1},\tau_{l_0+1},\ldots,\tau_l)$
and $(\underline{\tau}_1,\lambda,\underline{\tau}_2)
:=
(\tau_1,\ldots,\tau_{l_0-1},\lambda,\tau_{l_0+1},\ldots,\tau_l)$.
Finally, the degree of $S$ is
$k_1'+\ldots+k_l'+2\mu+l-2+k_0'+1=2\mu+const$.

Poincaré-Miranda theorem (see e.g. \cite{Kulpa}, p. 547 or \cite{OutereloRuiz}, pp. 152-153) helps to find a solution so that
$S$ vanishes at all prescribed $a_j$'s.
In detail, 
put $\cR:=[a_1,a_2]\times \ldots \times [a_{l_0-1},a_{l_0}]\times [0,1]\times[a_{l_0+1},a_{l_0+2}]\times\ldots\times[a_{l-1},a_l]$
and for  $j=1,\ldots,l$
let $f_j: \cR \rightarrow\bR$,
\begin{equation*}
f_j(\underline{\tau}_1,\lambda,\underline{\tau}_2):=
\int_{a_j}^{a_{j+1}} Z_1(t)
P_1(\lambda,\mu;t) \, 
R(\underline{\tau};t)
(t-x_0)^{k_0'}\, 
dt.
\end{equation*}
Now we verify the signs of these functions on opposite sides of $\cR$:
 if $j=1,\ldots,l$, $j\ne l_0$, then
$A_j:=\{(\tau_1,\ldots,\tau_{l_0-1},\lambda, \tau_{l_0+1},\ldots,\tau_l)\in\cR:
\tau_j=a_j\}$
and $B_j:=\{(\tau_1,\ldots,\tau_{l_0-1},\lambda, \tau_{l_0+1},\ldots,\tau_l)\in\cR:
\tau_j=a_{j+1}\}$
are the opposite sides.
We have to treat the case $j<l_0$ and the case $j>l_0$ separately.
If $(\underline{\tau}_1,\lambda,\underline{\tau}_2)\in A_j$, then 
$R(t)$ has the same sign all over $(a_j,a_{j+1})$ and 
$\sign f_j(\underline{\tau}_1,\lambda,\underline{\tau}_2)=
\sign R(t)(t-x_0)^{k_0'} 
=(-1)^{l-1-j+k_0'}=(-1)^{l-j}$ if $j<l_0$ and 
$\sign f_j(\underline{\tau}_1,\lambda,\underline{\tau}_2)=\sign R(t)=(-1)^{l-1-j}$ if $j>l_0$.
On the other 
side, 
if $(\underline{\tau}_1,\lambda,\underline{\tau}_2)\in B_j$,
then this means that we move $\tau_j$ from $a_j$ to $a_{j+1}$
hence the sign of $R(t)$ changes.
That is, 
the sign of $R(t)$ 
is the same as that of $f_j(\underline{\tau}_1,\lambda,\underline{\tau}_2)$,
hence 
if $j<l_0$, then
$\sign f_j(\underline{\tau}_1,\lambda,\underline{\tau}_2)=
\sign R(t)(t-x_0)^{k_0'} 
=(-1)^{l-j+1}$
and if $j>l_0$, then
$\sign f_j(\underline{\tau}_1,\lambda,\underline{\tau}_2)=(-1)^{l-j}$,
which shows the sign change in both cases (when $j=1,\ldots,l_0-1$ and when $j=l_0+1,\ldots,l$).

As regards $j=l_0$, 
we estimate $Z_1(t)$ and $R(t)$ first.
Let $\rho_1:=1/4 \min(a-a',x_0-a,b-x_0,b'-b)>0$.
Considering $Z_1(t)$, it is easy to see that there exists $C_3>0$ such that 
for all $t\in[\alpha-\rho_1,\alpha+\rho_1]\cup[\beta-\rho_1,\beta+\rho_1]$ we have 
$1/C_3\le Z_1(t)\le C_3$.
The family of 
possible  polynomials $R(\underline{\tau};t)$
also has this property:
there exists $C_4>0$ such that for 
any $(\underline{\tau}_1,\lambda,\underline{\tau}_2)\in\cR$,
and for any $t\in  [\alpha-\rho_1,\alpha+\rho_1]\cup[\beta-\rho_1,\beta+\rho_1]$ we have 
$1/C_4 \le |R(\underline{\tau};t)(t-x_0)^{k_0'}|\le C_4$.
Now
we need Nikolskii inequality
to give a lower estimate for the integral of $P_0$ near $\alpha$ and $\beta$.
Using that $\left\|P_0(\alpha,\mu;.)\right\|_{[\alpha-\rho_1,\alpha+\rho_1]}=P_0(\alpha)=1$ and $\deg(P_0)=2\mu$,
Nikolskii inequality (see e.g. \cite{MMR}, p. 498, Theorem 3.1.4.) 
yields that there exists $C_5>0$ independent of $\mu$ and $P_0$ such that 
\begin{equation*}
\int_{\alpha-\rho_1}^{\alpha+\rho_1}
P_0(\alpha,\mu;t) 
dt
= \int_{\alpha-\rho_1}^{\alpha+\rho_1} \left| P_0(\alpha,\mu;t) \right| dt 
\ge 
C_5  \frac{1}{\mu^2} 
\end{equation*}
with some $C_5>0$ depending on $\rho_1$ only
and we can easily obtain
\begin{equation}
\int_{\alpha-\rho_1}^{\alpha+\rho_1}
P_0(\alpha,\mu;t) 
Z_1(t) \left|R(\underline{\tau};t)
(t-x_0)^{k_0'}\right|
dt
\ge 
\frac{C_5}{C_3 C_4}
\frac{1}{\mu^2} 
\label{nikolapplied}
\end{equation}
as well. 
Moreover, for any $\lambda\in [0,1]$,
$\max_{[\alpha-\rho_1,\alpha+\rho_1]} P_1(.) \ge 1-\lambda$,
hence 
applying Nikolskii inequality (see e.g. \cite{MMR}, p. 498, Theorem 3.1.4.)  on these intervals,
\begin{equation*}
\int_{\alpha-\rho_1}^{\alpha+\rho_1}
P_1(\lambda,\mu;t) 
Z_1(t) \left|R(\underline{\tau};t)
(t-x_0)^{k_0'}\right|
dt
\ge 
\frac{C_5}{C_3 C_4}
\frac{1-\lambda}{\mu^2} 
\end{equation*}
and similarly for $[\beta-\rho_1,\beta+\rho_1]$.

We need an upper  estimate too.
If $t\in[a_0,a_{l+1}]$, $|t-\alpha|\ge \rho_1$,
then with $\rho_2:=1-\left(\frac{\rho_1}{c_2}\right)^2<1$
we can write 
\begin{equation*}
P_0(\alpha,\mu;t) 
\le \rho_2^\mu
\end{equation*}
and if 
$t\in H:= [a_0,\alpha-\rho_1]\cup[\alpha+\rho_1,\beta-\rho_1]\cup[\beta+\rho_1,a_{l+1}]$
then
\begin{multline}
P_0(\alpha,\mu;t) 
Z_1(t) 
\left| R(t) (t-x_0)^{k_0'}\right|
,
P_0(\beta,\mu;t) 
Z_1(t) 
\left| R(t) (t-x_0)^{k_0'}\right|
\le 
C_3 C_4 \rho_2^\mu
\label{faraway}
\end{multline}
and 
\begin{gather}
P_0(\alpha,\mu;t) 
Z_1(t) 
\left| R(t) (t-x_0)^{k_0'}\right|
\le 
C_3 C_4 \rho_2^\mu,
\ |t-\beta|\le\rho_1,
\label{farawayfrombeta}
\\
P_0(\beta,\mu;t) 
Z_1(t) 
\left| R(t) (t-x_0)^{k_0'}\right|
\le 
C_3 C_4 \rho_2^\mu,
\ |t-\alpha|\le\rho_1.
\label{farawayfromalpha}
\end{gather} 

Now we can investigate $f_{l_0}(.)$
on $A_{l_0}:=\{(\tau_1,\ldots,\tau_{l_0-1},\lambda, \tau_{l_0+1},\ldots,\tau_l)\in\cR:
\lambda=0\}$:
by \eqref{nikolapplied}
we can write
\begin{multline*}
\left|\int_{a_{l_0}}^{x_0}
P_0(\alpha,\mu;t)
Z_1(t) R(\underline{\tau};t)(t-x_0)^{k_0'}
dt
\right|
\\
\ge 
\int_{\alpha-\rho_0}^{\alpha+\rho_0}
P_0(\alpha,\mu;t)
Z_1(t) \left|R(\underline{\tau};t) (t-x_0)^{k_0'}\right| dt
\ge 
\frac{C_5}{C_3 C_4} 
\frac{1}{\mu^2} 
\end{multline*}
and by \eqref{faraway},
we can write
\begin{equation*}
\left|\int_{x_0}^{a_{l_0+1}}
P_0(\alpha,\mu;t)
Z_1(t) R(\underline{\tau};t) 
(t-x_0)^{k_0'}
dt
\right|
\le 
c_2 C_3 C_4 \rho_2^\mu.
\end{equation*}
These last two displayed estimates
show that $f_{l_0}(.)$ on $A_{l_0}$
has the same sign as $R(t)(t-x_0)^{k_0'}$ on $(a_{l_0},x_0)$
(that is,  
$(-1)^{l-l_0-1 +k_0'} 
=(-1)^{l-l_0}$) if $\mu$ is large ($\mu\ge \mu_1$).
Similarly, by replacing $\alpha$ with $\beta$,
we can say that
$f_{l_0}(.)$ on $B_{l_0}:=\{(\tau_1,\ldots,\tau_{l_0-1},\lambda, \tau_{l_0+1},\ldots,\tau_l)\in\cR:
\lambda=1\}$ 
has the same sign as $R(t)(t-x_0)^{k_0'}$ on $(x_0,a_{l_0+1})$
(that is, $(-1)^{l-l_0 +1}$), again if $\mu$ is large ($\mu\ge \mu_2$).
These two observations show that
on the opposite sides $A_{l_0}$
and $B_{l_0}$, $f_{l_0}(.)$ has different signs (since $k_0'$ is odd).
Obviously, all $f_j(.)$ functions are continuous.

Now the conditions of Poincaré-Miranda theorem are satisfied,
hence there exists $(\underline{\tau}_1,\lambda,\underline{\tau}_2)\in\cR$
such that  $f_j(\underline{\tau}_1,\lambda,\underline{\tau}_2)=0$ for all $j=1,\ldots,l$.
Fix these values and denote them
by the same letters in the rest of this proof.

Finally, in \eqref{Qrepr}, we choose
$C_1\in\bR$ so that $S(x_0)=1$, where actually we can write 
\begin{equation*}
\frac{1}{C_1}=
\int_{a_{l_0}}^{x_0}
P_1(\lambda,\mu;t) Z_1(t)
\,
R(\underline{\tau};t) 
(t-x_0)^{k_0'}
\,
dt
\end{equation*} 
and by knowing the sign of $R(\underline{\tau};.)$ over 
$(a_{l_0},x_0)$, $\sign C_1 = (-1)^{l-1-l_0 +k_0'}=(-1)^{l-l_0}$ and by \eqref{nikolapplied}, 
$|C_1|=O(\mu^2)$.

So $S$ is uniquely determined
and it has the following properties.
$S(a_j)=0$ for all $j=1,\ldots,l$,
hence by the key observation, \eqref{fdpz:zeros} holds.
By the normalization \eqref{fdpz:atxnull} is true.
\eqref{fdpz:datxnull} is also true,
because of \eqref{Qrepr}.
For simplicity, put 
\begin{equation*}
S_1(t):= C_1 Z_1(t) P_1(t) R(t) (t-x_0)^{k_0'}.
\end{equation*}
To see \eqref{fdpz:peaking},
\eqref{fdpz:high},
\eqref{fdpz:change},
and the first half of \eqref{fdpz:low}
(with $1$ in place of $\min(1,|Z(x)|)$)
first note that \eqref{faraway} 
implies that
\begin{equation}
\left|Z_1(t) P_1(\mu;t) R(t) (t-x_0)^{k_0'}\right|\le  C_3 C_4 \rho_2^\mu
\label{zoneponeest}
\end{equation}
when $t\in H=[a_0,\alpha-\rho_1]\cup[\alpha+\rho_1,\beta-\rho_1]\cup[\beta+\rho_1,a_{l+1}]$.
Moreover, let us remark that
\begin{equation}
|P_1(\mu;t)| \le \rho_2^\mu
\label{ponesupnormest}
\end{equation}
for $t\in H$.
Let us choose $\delta_1>0$ such that $0<\delta_1<-1/64 \log(\rho_2)$, 
hence for large $\mu$,
$\mu \ge \mu_3$, we have
\begin{equation*}
C_3 C_4 \rho_2^\mu 
\le 
\exp(-\delta_1 (64\mu)).
\end{equation*}
Now, if $\mu\ge\mu_4$ is large enough and 
using $|C_1|=O(\mu^2)$, 
we can write 
\begin{equation*}
\left|S_1(t)\right|\le 
C_1 C_3 C_4 \rho_2^\mu \le \exp(-\delta_1 (32\mu)), \quad t\in H.
\end{equation*}
Integrating this on 
$[a_1,x]$, 
$x\le\alpha-\rho_1$, we obtain for large $\mu$, $\mu\ge \mu_5$, that
\begin{equation*}
|S(x)|=
\left|\int_{a_1}^x S_1(t) dt\right| 
\le 
c_2 C_1 C_3 C_4 \rho_2^\mu 
\le \exp(-\delta_1(16\mu))
\end{equation*}
moreover this also holds when $x\in[a_0,a_1]$. 
If $x\in[\alpha+\rho_1,x_0]$, then
using that $S_1(t)\ge 0$ when $t\in[\alpha+\rho_1,x_0]$, we can write
\begin{multline*}
1-\exp(-\delta_1 (16\mu))
\le 
1-c_2 C_1 C_3 C_4 \rho_2^\mu
\le 
\int_{a_1}^{x_0} S_1(t) dt 
-
\int_{x}^{x_0} S_1(t)dt
\\=
S(x)
\le S(x_0)
=1.
\end{multline*}
Similarly when $x\in[x_0,\beta-\rho_1]$, $S_1(t)\le 0$ on 
$[x_0,\beta-\rho_1]$, hence
\begin{multline*}
1-\exp(-\delta_1 (16\mu))
\le 
1-c_2 C_1 C_3 C_4 \rho_2^\mu
\le 
\int_{a_1}^{x_0} S_1(t) dt 
+
\int_{x_0}^{x} S_1(t)dt
\\=
S(x)
\le S(x_0)
=1.
\end{multline*}
As for $[\beta+\rho_1,a_{l+1}]$,
we know that $|S_1(t)|\le C_1 C_3 C_4 \rho_2^\mu\le \exp(-\delta_1(32\mu))$, and $S(a_{l_0+1})=0$, 
so for $x\in[a_{l_0+1},a_{l+1}]$,
$S(x)=\int_{a_1}^x S_1(t) dt=\int_{a_{l_0+1}}^x S_1(t) dt$ 
and $|S(x)|\le c_2 C_1 C_3 C_4 \rho_2^\mu \le 
\exp(-\delta_1(16\mu))$.
For $x\in[\beta+\rho_1,a_{l_0+1}]$,
we know that
\begin{multline*}
S(x)=
\int_{a_1}^x S_1(t) dt= 
\int_{a_1}^{a_{l_0+1}} S_1(t)dt- 
\int_{x}^{a_{l_0+1}} S_1(t)dt
\\=
0+\int_x^{a_{l_0+1}} -S_1(t) dt=
\int_x^{a_{l_0+1}} \left|S_1(t)\right| dt
\le c_2 C_1 C_3 C_4 \rho_2^\mu
\le 
\exp(-\delta_1 (16\mu)).
\end{multline*}
These last four displayed estimates
show that \eqref{fdpz:high} and
first half of \eqref{fdpz:low} hold
since 
\begin{equation*}
\exp(-\delta_1 (16\mu))
\le \exp(-2\delta_1  (3\deg S))
\end{equation*}
if $\mu\ge \mu_6$ is large.
\eqref{fdpz:change} and \eqref{fdpz:peaking} are also true,
since $S'(.)=S_1(.)$ is nonnegative 
on $(a_{l_0},x_0)$ and is nonpositive on $(x_0,a_{l_0+1})$.

To establish the second half of \eqref{fdpz:low} (with $Z(x)$ in place of $\min(1,|Z(x)|)$),
we write (similarly to \eqref{zoneponeest})
\begin{multline*}
\left|S(x)\right| = \left| C_1
\int_{a_1}^x Z_1(t) \frac{P_1(t)}{\|P_1\|_H}
R(t) (t-x_0)^{k_0'} dt 
\right|
\|P_1\|_H
\\
\le |C_1| \int_{a_1}^x Z_1(t) \frac{|P_1(t)|}{\|P_1\|_H}
|R(t)| |t-x_0|^{k_0'} dt \|P_1\|_H
\\ 
\le 
|C_1| C_4 \int_{a_j}^x  Z_1(t) dt
\, \|P_1\|_H
\end{multline*}
where $x\in [a_j,a_{j+1}]$ and
$H= [a_0,\alpha-\rho_1]\cup[\alpha+\rho_1,\beta-\rho_1]\cup[\beta+\rho_1,a_{l+1}]$.
It is easy to see that 
\begin{equation*}
\frac{\int_{a_j}^x  Z_1(t) dt}{|Z(x)|}
\end{equation*}
has finite limit as $x\rightarrow a_j$
since $Z$ and $Z_1$ have zeros of order $k_j$
and $k_j'$ at $a$ respectively.
The same is true on the left hand side neighborhood of $a_j$.
Hence we see that
$\int_{a_1}^x Z_1(t) dt /|Z(x)|$ is bounded 
when $x\in H$, so,
using $\|P_1\|_H\le \exp(-\delta_1 32 \mu)$
coming from \eqref{ponesupnormest},
we obtain that the second half of \eqref{fdpz:low}
holds for large $\mu$, $\mu\ge \mu_7$.

To 
fulfill
\eqref{fdpz:nonneg},
consider $Q:=S^2$.
Then, the degree of $Q$ is  $2(k_1'+\ldots+k_l'+l-2+k_0'+2\mu +1)=4\mu+const$.
By squaring $S$ defined in 
\eqref{Qrepr},
it is easy to see that
\eqref{fdpz:atxnull}, \eqref{fdpz:datxnull}, \eqref{fdpz:peaking}, \eqref{fdpz:low}, \eqref{fdpz:zeros}
and \eqref{fdpz:change} are preserved,
and actually, 
\eqref{fdpz:high} too: 
\begin{equation*}
\left(1-\exp(-2\delta_1  (3\deg S))\right)^2
\ge
1-\exp(-2\delta_1  \deg Q)
\end{equation*}
since $2\exp(-2\delta_1  3 \deg S)-
\exp(-4\delta_1 3 \deg S)\le 
\exp(-2\delta_1 \deg Q)$ 
if $\deg S$ is large  (that is, if $\mu\ge\mu_8$).

Finally, we have a sequence of polynomials for particular degrees.
The basic idea to use the same polynomial for larger degree works now,
because of the following.
Put $m_1(m):= \max\{m_1: m_1=4\mu+ 2(k_1'+\ldots+k_l'+l-2+k_0'+1), m_1\le m, \mu\in\bN\}$.
For general $m\in\bN$,
replacing the error term for $m$ from $m_1(m)$ 
brings in a factor $\exp(-2\delta_1 m)/\exp(-2\delta m_1(m))$
which can be estimated as
\begin{equation*}
\limsup_{m\rightarrow\infty} \exp(-2\delta_1 m)/\exp(-2\delta_1 m_1(m))=\exp(-2\delta_1 const)<1,
\end{equation*}
where $const$ is actually $2(k_1'+\ldots+k_l'+l-2+k_0'+1)$.
Hence, if $\mu\ge \mu_9$ is large, then
\begin{equation*}
\exp(-2\delta_1 m_1(\deg Q))
\le 
\exp(-\delta_1 \deg Q)
\end{equation*}
which finishes the proof.
\end{proof}

Remark:
Note that (the second half of) \eqref{fdpz:low} implies \eqref{fdpz:zeros}.

We need the following 
trigonometric form 
of fast decreasing polynomials.
In the proof we use so-called half-integer 
trigonometric polynomials $\sum_{j=0}^{n} a_j \cos((j+1/2)t) + b_j \sin((j+1/2)t)$.
They are natural in this context,
see, e.g.~ the product representation \cite{BorweinErdelyi}, p.~10, 
or Videnskii's original paper 
\cite{VidenskiiHalfInt},
or the paper \cite{NT2013}.

\begin{theorem}
\label{thm:fdtp}
Let $t_0,\alpha,\beta,\alpha',\beta'\in(-\pi,\pi)$  
be such that
$-\pi<\alpha'<\alpha<t_0<\beta<\beta'<\pi$ and
$\alpha_1,\ldots,\alpha_l\in(-\pi,\pi)\setminus [\alpha',\beta']$
be with the corresponding positive integer powers
$k_1,\ldots,k_l$.
Put $Z(t):=\prod_{j=1}^l\left|\sin \frac{t-\alpha_j}{2}\right|^{k_j}$.

Then there exists $\delta_1>0$ such that for all large $m$ there exists
a trigonometric polynomial $Q_m$ with degree at most $m$ such that
\begin{gather}
Q_m(t_0)=1,
\label{fdtp:atxnull}
\\
0\le Q_m(t)<1 \mbox{ for } t\in[-\pi,\pi), t\ne t_0,
\label{fdtp:nonneg}
\\
Q_m^{(k)}(\alpha_j)=0, \quad
j=1,\ldots,l, \  k=0,1,\ldots,k_j,
\label{fdtp:zeros}
\\
Q_m(t)\le 
\min(1,|Z(t)|)
\exp(-\delta_1 m)
\mbox{ for } t\in[-\pi,\pi]\setminus (\alpha',\beta'),
\label{fdtp:low}
\\
|Q_m(t)-1|\le \exp(-\delta_1 m)
\mbox{ for } t\in [\alpha,\beta],
\label{fdtp:high}
\\
Q_m(t) \mbox{ is strictly monotone on } [\alpha',\alpha] \mbox{ and on }
[\beta,\beta'].
\label{fdtp:change}
\end{gather}
\end{theorem}
\begin{proof}
Briefly, we use similar idea as in the previous proof (Theorem \ref{prop:fdpz}),
but there are lots of differences.

First, we introduce 
the intervals between the neighboring $\alpha_j$'s as follows
using the ordering of $\alpha_j+\epsilon_j 2\pi$, $j=1,\ldots,l$,
and $\epsilon_j=0$ if $\alpha_j>\beta'$ and $\epsilon_j=1$ otherwise.
Let $I_j$'s, $j=1,\ldots,l-1$ denote the closed intervals  such that endpoints are the $\alpha_j+\epsilon_j 2\pi$'s 
and they are disjoint except for the endpoints, and they are ordered from left to right
(that is, if $t_1\in I_j$ and $t_2\in I_k$ and $j\le k$, then
$t_1\le t_2$).
Denote the left endpoint of $I_1$ by $\alpha_*$,
and the right endpoint of $I_{l-1}$ by $\alpha^*$,
that is, $\alpha_*$ and $\alpha^*$ are the minimum  and maximum of $\alpha_j+\epsilon_j 2\pi$'s respectively.
Put $I_0:=[\alpha^*-2\pi,\alpha_*]$, this way $I_0,I_1,\ldots,I_{l-1}$ cover an interval of length $2\pi$ and $t_0\in I_0$, $[\alpha',\beta']\subset I_0$.
Note that $I_j$'s are not necessarily subsets of $(-\pi,\pi)$.

\begin{figure}
\begin{center}
\includegraphics[keepaspectratio,width=\textwidth]{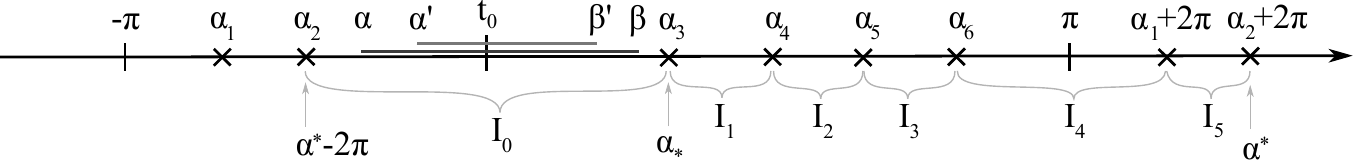}
\end{center}
\caption{Prescribed zeros and intervals in the trigonometric case}
\end{figure}

We define
\begin{gather*}
\tilde{Z}_1(t):=\prod_{j=1}^{l} \left(\sin \frac{t-\alpha_j}{2}\right)^{k_j'},
\quad
\tilde{R}(\underline{\tau}; t)=\tilde{R}(t):= \prod_{j=1}^{l-1}\sin\frac{t-\tau_j}{2},
\\
\tilde{P}_0(t)=\tilde{P}_0(a,\mu;t):=
\left(\cos\frac{t-a}{2}\right)^{2\mu},
\\
\tilde{P}_1(a,b,\lambda,\mu;t)=
(1-\lambda)
\tilde{P}_0(a,\mu;t)
+
\lambda
\tilde{P}_0(b,\mu;t)
\end{gather*}
where  
$k_j'=k_j$ if $k_j$ is even and
$k_j'=k_j+1$ if $k_j$ is odd,
for $j=1,\ldots,l$,
and $\tau_j\in I_j$, $j=1,\ldots,l-1$, 
and $\alpha'<a<\alpha<\beta<b<\beta'$, 
$a:=(\alpha+\alpha')/2$, $b:=(\beta+\beta')/2$,  
and 
$\lambda\in[0,1]$.
We also put $k_0'=k_0$ if $k_0$ is odd
and $k_0'=k_0+1$ if $k_0$ is even; and 
$\underline{\tau}:=(\tau_1,\ldots,\tau_{l-1})$.
As above, if some of the parameters are fixed or unimportant in the current consideration, then we leave them out, e.g.
$\tilde{P}_0(t)=\tilde{P}_0(a,\mu;t)$ and
$\tilde{P}_1(t)=\tilde{P}_1(\mu;t)=\tilde{P}_1(\lambda,\mu;t)=\tilde{P}_1(a,b,\lambda,\mu;t)$.

Some immediate properties are the following: 
$\tilde{Z}(t)$, $\tilde{P}_0(t)$ and $\tilde{P}_1(t)$ are nonnegative trigonometric polynomials.
If $l$ is even, then $\tilde{R}(t)$ is a half-integer trigonometric polynomial, 
if $l$ is odd, then it is a trigonometric polynomial (with degree $(l-1)/2$).

Consider
\begin{equation*}
\tilde{S}_1(t)
:=
\tilde{Z}_1(t)\, \tilde{P}_1(\mu,\lambda; t)\, \tilde{R}(\underline{\tau};t) \left(\sin\frac{t-t_0}{2}\right)^{k_0'}
\end{equation*}
which is a trigonometric polynomial if $l$ is 
even
and is a half-integer trigonometric polynomial if $l$ is 
odd. 
We need 
\begin{equation*}
\tilde{S}_2(t):=
\begin{cases}
 \tilde{S}_1(t) , &
\mbox{ if } l \mbox{ is even,}
\\
 \tilde{S}_1(t) \cos\frac{t-(\alpha^*-\pi)}{2} , 
&
\mbox{ if } l \mbox{ is odd}
\end{cases}
\end{equation*}
which is a trigonometric polynomial in both cases.

Now we would like to integrate $\tilde{S}_1(.)$ and get a trigonometric polynomial too. 
To do this, we use Poincaré-Miranda theorem,
as in the proof of Theorem \ref{prop:fdpz}.
Consider the rectangle $\mathcal{R}:=[0,1]\times I_1\times I_2\times\ldots\times I_{l-1}$ and
$(\lambda,\underline{\tau})=(\lambda,\tau_1,\ldots,\tau_{l-1})\in\mathcal{R}$.
We use the functions 
\begin{equation*}
f_j(\lambda,\underline{\tau})
:= 
\int_{I_j} \tilde{S}_2(\lambda,\underline{\tau},\mu;t) dt,
\  j=0,1,\ldots,l-1.
\end{equation*}
Note that $\sin\frac{t-t_0}{2}$ is negative on $(\alpha^*-2\pi,t_0)$ and is positive on $(t_0,\alpha^*)$, 
$\cos\frac{t-(\alpha^*-\pi)}{2}$ is positive on 
$(\alpha^*-2\pi,\alpha^*)$ but 
it introduces an extra zero at $\alpha^*$.
It can be verified same way as in the proof of Theorem \ref{prop:fdpz} that there are sign changes in
$f_0$ as $\lambda$ changes from $0$ to $1$,
and in $f_j$ as $\tau_j$ goes from the left endpoint of $I_j$ to the right endpoint of $I_j$.

Poincaré-Miranda theorem shows that there are
particular $\lambda\in[0,1]$, $\tau_1\in I_1,\ldots,\tau_{l-1}\in I_{l-1}$ such that all the $f_j$'s are zero; fix this solution and denote it by $\lambda,\tau_1,\ldots,\tau_{l-1}$ in the rest of this proof.
Summing up these integrals for all $j=0,1,\ldots,l-1$,
we also obtain that $\int_{\alpha^*-2\pi}^{\alpha^*} \tilde{S}_2(t) dt=0$.

Put
\begin{equation*}
\tilde{S}(t):=
\int_{\alpha_*}^t C_1 \tilde{S}_2(\tau) d\tau
\end{equation*}
where $C_1$ will be chosen later 
(like in the proof of Theorem \ref{prop:fdpz}).
In both cases ($l$ is even or odd), the integrand is a real trigonometric polynomial.
Since the integral of $\tilde{S}_2(t)$  over $[\alpha^*-2\pi,\alpha^*]$ is $0$,
$\tilde{S}(t)$ is also a trigonometric polynomial.
$C_1$ can be chosen so that
\begin{equation*}
\int_{\alpha_*}^{t_0} C_1 \tilde{S}_2(t) dt=1
\end{equation*}
holds.
The properties \eqref{fdtp:atxnull},
\eqref{fdtp:zeros}, \eqref{fdtp:low},
\eqref{fdtp:high}
and \eqref{fdtp:change}
can be verified same way as in the 
proof of Theorem \ref{prop:fdpz}.
A key tool was the Nikolskii inequality for algebraic polynomials and it should be replaced with the similar inequality for trigonometric polynomials, which is again due to Nikolskii (see, e.g~\cite{MMR}, p.~495, Theorem 3.1.1). 
Again, squaring $\tilde{S}$, we can construct the trigonometric polynomial which also satisfies 
\eqref{fdtp:nonneg}.

\end{proof}

\subsection{Rough Markov- and Bernstein-type inequalities}

The following two propositions have rather simple proofs,  they may be known, but  we could not find reference for them.

\begin{proposition}
Let $I\subset (-\pi,\pi)$ be a closed set consisting  of finitely many disjoint intervals such that 
none of them is a singleton  and $k$ be a positive integer.
Then there exists $C=C(I,k)>0$ such that
for all trigonometric polynomial $T_n$ with degree $n$,
we have
\begin{equation}
\label{roughMarkov}
\left\|T_n^{(k)}\right\|_I
\le 
C n^{2k}
\left\|T_n\right\|_I.
\end{equation}
\end{proposition}
This immediately follows from iterating Videnskii's inequality on each component (maximal subinterval) of $I$. For Videnskii's inequality, see
\cite{BorweinErdelyi}, p.~243 (Exercise E.19 part c]) or \cite{Videnskii}.

We also need a rough Bernstein-type inequality for higher derivatives of
trigonometric polynomials.

\begin{proposition}
Let $I\subset (-\pi,\pi)$ be again a closed set consisting  of finitely many disjoint intervals such that 
none of them is a singleton  and $k$ be a positive integer.
Fix a closed set $I_0\subset \interior I$
(subset of the one dimensional interior of $I$).
Then
there exists $C=C(I,I_0,k)>0$ such that
for all trigonometric polynomial $T_n$ with degree $n$,
we have for $t\in I_0$
\begin{equation}
\label{roughBernstein}
\left|T_n^{(k)}(t)\right|
\le 
C n^{k}
\left\|T_n\right\|_I.
\end{equation}
\end{proposition}
This again, immediately follows from applying Videnskii's inequality
(see \cite{BorweinErdelyi}, p.~243, E.19 part b]) iteratively
on the component (say $I_0^+$) of $I$ containing $I_0$
and finally using $\|T_n\|_{I_0^+}\le 
\|T_n\|_{I}$.

\subsection{Asymptotically sharp Markov-type inequality}

\begin{theorem}
\label{highertrigMarkov}
Let $E\subset (-\pi,\pi)$ be a compact set satisfying \eqref{intervalcondition}.
Then for any trigonometric polynomial $T_n$ with degree $n$, we have
\begin{equation}\label{eq:Markov}
\left\|T_n^{(k)}\right\|_{[a-\rho,a]}
\le 
(1+o(1)) n^{2k} \Omega(E_\bT,e^{i a})^{2k} \frac{8^k \pi^{2k}}{(2k-1)!!} 
\left\|T_n\right\|_{E}
\end{equation}
where $o(1)$ is an error term that tends to $0$ as $n\rightarrow\infty$, 
depends on $E$ and $a$, but
it is independent of $T_n$.
This inequality is sharp, 
that is, there is a sequence of trigonometric 
polynomials $T_n$, 
$n=1,2,\dots$, such that $\deg T_n=n$ and
\begin{equation}\label{eq:Markovacc}
\left|T_n^{(k)}(a)\right|
\ge 
(1-o(1))n^{2k} \Omega(E_\bT,e^{i a})^{2k} \frac{8^k \pi^{2k}}{(2k-1)!!} 
\left\|T_n\right\|_{E},
\end{equation}
where $o(1)\rightarrow 0$ is an error term depending on $E$ and $n$.
\end{theorem}

\begin{proof}
The proof 
of \eqref{eq:Markov} 
is divided into five steps
and then \eqref{eq:Markovacc} will be established. 

First step. 
We prove the assertion when
$E$ is a T-set, and $T_n$ is polynomial of the defining polynomial $U_N$ for this set.
That is, $E=\{t\in (-\pi,\pi): \  \left|U_N(t)\right|\le 1\}$ (as in \eqref{Tsetdef})
and there is a real, algebraic polynomial $P$ such that $T_n(t)=P(U_N(t))$.
We may assume that $U_N(a)=1$ (we know that $|U_N(a)|=1$).

Now we use Fa\`a di Bruno's formula 
\eqref{faadibruno}.
Note that, in our setting $f=P$ (outer function) and $g=U_N$ (inner function),
hence the product is independent of $P$ and $n$ (and $T_n$ too).
Hence we reorder the terms   decreasingly:
\begin{equation}\label{kderiv}
\left(P\circ U_N\right)^{(k)}(a)
=
P^{(k)}(1) \left(U_N'(a)\right)^{k}
+ \ldots
\end{equation}
where in the remaining terms only $P^{(k-1)}(1)$, $P^{(k-2)}(1)$, ... $P'(1) $ occur.
There are finitely many remaining terms and by \eqref{roughMarkov}, 
they grow like $n^{2k-2}$ as $n\rightarrow\infty$.
As for the first term, we can use the classical V.~Markov inequality (see e.g. \cite{BorweinErdelyi}, p.~254) and $\left\|P\right\|_{[-1,1]}=\left\|T_n\right\|_E$,
hence with $d:=\deg(P)$, 
\begin{equation*}
\left|P'(1)\right|\le
\frac{
d^2 (d^2-1)\ldots \left(d^2-(k-1)^2\right)
}{
(2k-1)!!
}
\left\|T_n\right\|_E
\le d^{2k} \frac{1}{(2k-1)!!}
\left\|T_n\right\|_E
\end{equation*}
where actually 
\begin{equation}\label{eq:limitbeh}
\frac{
d^2 (d^2-1)\ldots \left(d^2-(k-1)^2\right)
}{
d^{2k}
}
\rightarrow 1
\end{equation}
as $n\rightarrow\infty$ (which is equivalent 
to $d\rightarrow\infty$).

As for $U_N'(a)$, we use the density of the equilibrium measure, 
more precisely formula (3.21) from \cite{IMRN} (and $a=a_{2j_0})$, hence
\begin{equation}\label{eq:derofU}
|U_N'(a)|=
2 N^2 
\frac{\prod_{l=1}^{m} 
\left|e^{i a}-e^{i \tau_l}\right|^2
}{
\prod_{l=1,\ldots,2m,l\ne 2j_0 }
\left|e^{i a}-e^{i a_l} \right|
}
= 2 N^2 M_{a,k}
= 
8 \pi^2 N^2 \Omega(E_\bT, e^{ia})^2. 
\end{equation}

Putting these together:
\begin{equation*}
\left|T_n^{(k)}(a)\right|
\le 
(1+o(1))8^k \pi^{2k} \frac{1}{(2k-1)!!}
\Omega(E_\bT,e^{ia})^{2k} \; 
n^{2k} 
\left\|T_n\right\|_E.
\end{equation*}

Now we extend the previous inequality from $a$ to $[a-\rho,a]$ (as in \eqref{eq:Markov}).
Basically we use the smaller growth of the 
rough Bernstein-type inequality \eqref{roughBernstein} and the continuity of $U_N'$.
For any $\varepsilon>0$,
we can select $\eta>0$ such that $[a-\eta, a]\subset E$ 
and for $t\in [a-\eta,a]$ it is true that 
\begin{equation*}
|U_N'(t)|\le (1+\varepsilon)|U_N'(a)|=(1+\varepsilon)8 \pi^2 N^2 \Omega(E_\bT,e^{ia})^2.
\end{equation*}
Then for $t\in [a-\eta,a]$ we get from (\ref{kderiv}) and again from (\ref{roughMarkov})
that 
\begin{equation}
\left|T_n^{(k)}(t)\right|
\le 
(1+o(1))(1+\varepsilon)8^k \pi^{2k} \frac{1}{(2k-1)!!}
\Omega(E_\bT,e^{ia})^{2k} \; 
n^{2k} 
\left\|T_n\right\|_E.
\label{Bernnearendpoint}
\end{equation}
Now, on $[a-\rho,a-\eta]$ (if not empty),
we can use the rough Bernstein-type inequality \eqref{roughBernstein}, hence
we obtain an upper estimate for $T^{(k)}(t)$ which has growth order $n^k$,
which is smaller than $n^{2k}$, 
the growth order of the Markov factor.
So if $n$ is large (depending on $\varepsilon$), then 
\eqref{Bernnearendpoint} holds for
$t\in[a-\rho,a-\eta]$ too.
Now letting $\varepsilon\rightarrow 0$
appropriately,
(\ref{eq:Markov}) follows
for $T_n(.)=P(U_N(.))$ as $d=\deg(P)\rightarrow \infty$.

\medskip

Second step. 
Now we establish \eqref{eq:Markov} when $E$ is a T-set 
and $T_n$ is arbitrary trigonometric polynomial.
We use symmetrization here (see, \cite{Totik2001} pp.~151-152
and \cite{IMRN}, pp.~2997-2998, including Lemma 3.2) 
and fast decreasing trigonometric polynomials (see Subsection \ref{sec:fastdecrtrig}).
In this step we work in a smaller neighborhood of $a$, i.e. on $[a-\rho_0,a]$ where $\rho_0<\rho$ is defined later.

Let 
$j_0$ correspond to the interval in which $a$ is. 
More precisely,
since $E$ is a T-set in this case, 
there are $2N$ disjoint, open intervals
such that $U_N$ maps these intervals to $(-1,1)$ in a bijective way.
Let us label them by $E_j=(\alpha_{2j-1},\alpha_{2j})$ where $-\pi<\alpha_1<\alpha_2\le\alpha_3<\alpha_4\le\ldots\le\alpha_{2N-1}<\alpha_{2N}<\pi$.
Hence $a\in[\alpha_{2j_0 -1},\alpha_{2j_0}]$ and by \eqref{intervalcondition}, $a=\alpha_{2j_0}$.
Put $\rho_0:=1/4\min(\alpha_{2j_0}-\alpha_{2j_0-1},\alpha_{2j_0+1}-\alpha_{2j_0},\rho,\pi/4)$.

We also need the following facts on T-sets.
Since $U_{N}(.)$ is $2N$-to-$1$ mapping,
we need its restricted inverses.
Let $U_{N,j}^{-1}(t)$ be the inverse of $U_N$ restricted to $[\alpha_{2j-1},\alpha_{2j}]$
and put $t_j(t)=t_j:=U_{N,j}^{-1}(U_N(t))$.
Obviously, $t_j$ is $C^\infty$ on $\cup_{j=1}^N (\alpha_{2j-1},\alpha_{2j})$
and now we give estimates for the $l$-th derivative of $t_j(t)$, especially,
as $t$ approaches $a$.
Similarly, as in \cite{TZ},
if $l=1$ or $l=2$, then
\begin{gather*}
\frac{d t_j}{dt}=\frac{d}{dt}
U_{N,j}^{-1}(U_N(t))
=\frac{U_N'(t)}{U_N'(U_{N,j}^{-1}(U_{N}(t)))}
=\frac{U_N'(t)}{U_N'(t_j)},
\\
\frac{d^2}{dt^2}
U_{N,j}^{-1}(U_N(t))
=
\frac{-\left(U_N'(t)\right)^2 U_N''(U_{N,j}^{-1}(U_{N}(t)))}{\left(U_N'(U_{N,j}^{-1}(U_{N}(t)))\right)^3}
+
\frac{U_N''(t)}{U_N'(U_{N,j}^{-1}(U_{N}(t)))} \notag
\\ =
\frac{-(U_N'(t))^2 U_N''(t_j)}
{\left(U_N'(t_j)\right)^3} 
+
\frac{U_N''(t)}{U_N'(t_j)} \notag 
\end{gather*}
and for general $l$, Fa\`a di Bruno's formula \eqref{faadibruno} implies
that there is a universal polynomial $Q_{l}$
(independent of $U_N$, depending on $l$ only)
which is a polynomial in $U_N^{(k)}(t)$ and $U_N^{(k)}(t_j)$
$k=1,\ldots,l$,
that is $Q_l=Q_l(\ldots,U_N^{(k)}(t),\ldots, U_N^{(k)}(t_j),\ldots)$
such that
\begin{equation}
\frac{d^l}{dt^l}
U_{N,j}^{-1}(U_N(t))
=
\frac{Q_l}{\left(U_N'(t_j)\right)^{2l-1}}.
\label{dltj}
\end{equation}
Here, $Q_l$ is independent of $n$ and $T_n$, hence $|Q_l|\le C$ for some $C=C(k,U_N) >0$.

Moreover, we need to estimate
$|U_N'(t_j)|$ as $t\rightarrow a$
and we split the argument into two cases. 
If $j$ is such that $a_j\in \interior E$, that is, $U_N'(a_j)=0$, and using that
all the zeros of $U_N$ are simple,
we can infer that $U_N''(a_j)\ne 0$,
so $|U_N'(t_j)|\ge O(|t_j -a_j|)$.
On the other hand, if $j$ is such that $a_j\in E\setminus \interior E$,
that is, $U_N'(a_j)\ne 0$, then simply 
$U_N'(t_j)\approx U_N'(a_j)$.
Hence, in any case
\begin{equation}
\left|U_N'(t_j)\right|
\ge
O(|t_j -a_j|).
\label{UNprimelower}
\end{equation}

\smallskip

For an arbitrary polynomial $T_n$ 
consider $V_n(t)=L_{\sqrt{n}}(t)T_n(t)$, 
where $L_{\sqrt{n}}(.)$ denotes the fast decreasing polynomial which has the following properties.
$L_{\sqrt{n}}(.)$ has degree at most $\sqrt{n}$, it is a fast decreasing trigonometric polynomial and peaking at $a$ very smoothly (that is, $L_{\sqrt{n}}(a)=1$ and $L_{\sqrt{n}}^{(j)}(a)=0$, 
$j=1,2,\ldots,2k^2$), 
$L_{\sqrt{n}}(.)$ is approximately $1$ on $[a-\rho_0,a+\rho_0]$ and  is approximately $0$ outside $[a-2\rho_0,a+2\rho_0]$
and vanishes at the other extremal points of $U_N$ 
up to order $2k^2$ (that is, if $U_N(t)=\pm 1$, $t\ne a$, then $L_{\sqrt{n}}^{(j)}(t)=0$, $j=0,1,\ldots,2k^2$). 
Such polynomial $L(.)=L_{\sqrt{n}}(.)$ exists
because of Theorem \ref{thm:fdtp}.

For simplicity, put $W(t):=\prod_j \left(\sin\frac{t-\alpha_j}{2}\right)^{2k}$
where $j=1,\ldots,2N$, $j\ne j_0$.
This $W$ is a nonnegative trigonometric polynomial and has sup norm at most $1$.
There is another trigonometric polynomial $Y(.)$ such that
\begin{equation*}
L(t)= Y(t) W^k(t).
\end{equation*}
The sup norm of $Y$ over $[-\pi,a-\rho_0]\cup[a+\rho_0,\pi]$ can be estimated using
\eqref{fdtp:low}  with $W^k$ in place of $Z$.
Hence, for $t\in [-\pi,a-\rho_0]\cup[a+\rho_0,\pi]$
\begin{equation*}
|Y(t)|= \left|\frac{L(t)}{W^k(t)}\right|
\le 
\min\left(\frac{1}{W^k(t)},1\right)
\exp\left(-(\deg L) \delta_1\right).
\end{equation*}
Differentiating $L(.)$  $j$-times, $j=0,1,\ldots,k$ we write
\begin{equation}
L^{(j)}(t)= \sum_{l=0}^j
\binom{j}{l} Y^{(j-l)}(t) \left(W^k\right)^{(l)}(t).
\label{ljderivwithw}
\end{equation}
Here $\left(W^k\right)^{(l)}(t)=W(t)\cdot \ldots$
where $W(t)$ is multiplied with other terms
depending on $W,W',\ldots,W^{(l)}$, $k$ and $\alpha_j$'s only,
and it is independent from $n$ and $T_n$.
As regards $Y^{(j-l)}(t)$, we can use Videnskii's inequality for 
$Y(.)$ on  $[-\pi,a-\rho_0]\cup[a+\rho_0,\pi]$ (which is actually an interval on the torus), so 
there exists a $C>0$ such that
for all $t\in [-\pi,a-2\rho_0]\cup[a+2\rho_0,\pi]$
and all $l=0,1,\ldots,j$
\begin{equation}
\left|Y^{(j-l)}(t)\right|
\le 
C
\left(\deg Y\right)^{j-l}
\exp\left(-(\deg L) \delta_1\right).
\end{equation}
Summing up these estimates as in \eqref{ljderivwithw},
we can write with $\deg L\le  \sqrt{n}$
\begin{equation}
\left| L^{(j)}(t)\right|
\le 
C W(t)
n^{j/2} \exp\left(- \sqrt{n} \delta_1\right)
\label{LjestW}
\end{equation}
where $C>0$ is independent of $n$ and $T_n$
and $t\in [-\pi,a-2\rho_0]\cup[a+2\rho_0,\pi] $.

This $V_n$ has degree at most $n+\sqrt{n}$ 
and satisfies
\begin{equation}
\left.
\begin{aligned}
\left\|V_n\right\|_{E}
&\le
\left\|T_n\right\|_{E},
\\
V_n(t) &= 
\left(1+O(\beta^{\sqrt{n}})\right)
T_n(t) \mbox{ for } t\in[a-\rho_0,a],\\
\left|V_n(t)\right| 
&= O(\beta^{\sqrt{n}}) \left\|T_n\right\|_{E}
\mbox{ for } t\in E \setminus [a-2\rho_0,a]
\end{aligned}
\qquad \right\}
\label{estimatesforVn}
\end{equation}
where $\beta=\exp(-\delta_1)<1$.

Now, (by Leibniz formula), for all $l=1,\ldots,k$
\begin{equation}
V_n^{(l)}(t)-T_n^{(l)}(t)
=
\left( L_{\sqrt{n}}(t) -1
\right) T_n^{(l)}(t)
+ \sum_{j=1}^l
\binom{l}{j}
L_{\sqrt{n}}^{(j)}(t) 
T_n^{(l-j)}(t).
\label{leibniz}
\end{equation}
Using the rough Markov-type inequality \eqref{roughMarkov},
there exists a constant $C=C(E,k)>0$ 
such that for all $1\le j \le k$, $t\in E$
\begin{gather*}
\left|L_{\sqrt{n}}^{(j)}(t)\right|
\le C \sqrt{n}^{2j} 
\left\|L_{\sqrt{n}}\right\|_E
= C n^j,
\\
\left|T_n^{(j)}(t)\right|
\le C n^{2j} \left\|T_n\right\|_E
\end{gather*}
and if $t\in E\setminus [a-2\rho,a]$, 
then applying \eqref{roughMarkov} for $L_{\sqrt{n}}$ on $E\setminus [a-2\rho,a]$,
we can write
\begin{equation}
\left|L_{\sqrt{n}}^{(j)}(t)\right|
\le C \sqrt{n}^{2j} 
\left\|L_{\sqrt{n}}\right\|_{E\setminus [a-2\rho,a]}
=
C n^j \beta^{\sqrt{n}}.
\label{ljestimate}
\end{equation}
These imply that for $l=1,\ldots,k$
\begin{equation}
\left|
V_n^{(l)}(t)-T_n^{(l)}(t)
\right|
=O\left(n^{2l}\beta^{\sqrt{n}}
+ n^{2l-1}\right)
\left\|T_n\right\|_E,
\  t\in[a-\rho,a]
\label{vnktnkcompare}
\end{equation}
and
\begin{equation*}
\left|
V_n^{(l)}(t)
\right|
=
O\left(n^{2l} \beta^{\sqrt{n}}\right)
\left\|T_n\right\|_E,
\  t\in E\setminus[a-2\rho,a].
\end{equation*}

Define the "symmetrized" polynomial as
\begin{equation*}
T^*(t) := \sum_{j=1}^N V_n\left(t_j\right).
\end{equation*}
This $T^*$ will be algebraic polynomial of $U_N(.)$, see Lemma 3.2 in \cite{IMRN},  and $\deg(T^*)\le n+\sqrt{n}=(1+o(1))n$.

Now we compare $(T^*)^{(k)}(t)$ with $T_n^{(k)}(t)$  when $t\in[a-\rho,a]$.
If $j=j_0$, that is,
$t_j=t$, then $V_n(t_j)=V_n(t)$, and we can apply \eqref{vnktnkcompare}  (when $l=k$).  
If $j\ne j_0$, then we would like to show that $\left|\frac{d^k}{dt^k} V_n(t_j)\right|$ is small.
We use \eqref{leibniz} first so
\begin{equation}
\left|\frac{d^k}{dt^k} V_n(t_j)\right|
\le
\sum_{l=0}^k \binom{k}{l}
\left|\frac{d^l}{dt^l} 
L_{\sqrt{n}}
\left(  
U_{N,j}^{-1}(U_N(t))
\right)
\right|
\left|\frac{d^{k-l}}{dt^{k-l}} 
T_n
\left(  U_{N,j}^{-1}(U_N(t)) 
\right)
\right|
\label{dkvntj}
\end{equation}
which we continue later.
For the second factor, we use \eqref{faadibruno} again with similar groupings of the terms as in \eqref{kderiv},
because the first term involves $\frac{d^{k-l}}{dt^{k-l}}T_n$ 
(at $t_j$) and all the other terms involve lower derivatives of $T_n$.
So we can write, with the help of \eqref{roughMarkov},
and \eqref{dltj}, \eqref{UNprimelower}
\begin{multline*}
\left|\frac{d^{k-l}}{dt^{k-l}} 
T_n
\left(  U_{N,j}^{-1}(U_N(t)) 
\right)
\right|
\le 
\left|T_n^{(k-l)}(t_j)\right| 
\left|\frac{d}{dt} U_{N,j}^{-1}(U_N(t))
\right|^{k-l}
+\left|\ldots\right|
\\\le
C n^{2k-2l} 
\frac{1}{
\left|t_{j}-a_j\right|^{2(k-l)-1}} 
 \left\|T_n\right\|_{E}.
\end{multline*}

Now we use the zeros of $L_{\sqrt{n}}(.)$ 
(and $W(t)$) to 
get rid of the factors 
$1/\left|t_j-a_j\right|^{2(k-l)-1}$.
To estimate the first factor on rhs of \eqref{dkvntj}, 
we use \eqref{faadibruno} for $L_{\sqrt{n}}(.)$ and $U_{N,j}^{-1}(U_N(t))$
with \eqref{ljestimate} (since $t_j\not\in[a-2\rho,a]$)
and \eqref{LjestW}. 
Hence 
\begin{multline*}
\left|\frac{d^l}{dt^l} 
L_{\sqrt{n}}
\left(  
U_{N,j}^{-1}(U_N(t))
\right)
\right| 
\le C \sqrt{n}^{2l}
\beta^{\sqrt{n}} |W(t_j)| 
\,
n^{2l} 
\frac{1}{
\left|t_{j}-a_j\right|^{2l -1}} 
\left\|T_n\right\|_{E}
\\
= 
C n^{3l}
\beta^{\sqrt{n}}  
\frac{|W(t_j)|}{\left|t_{j}-a_j\right|^{2l -1}}
\left\|T_n\right\|_{E}
\end{multline*}
and using that  $a_j$ is a zero of $W$ (of order $k$),
the fraction $|W(t_j)|/\left|t_{j}-a_j\right|^{2l -1}$ is actually bounded.

Multiplying together the last two 
displayed estimates and using that 
$|W(t_j)|/|t_j-a_j|^{2k}$ is bounded
(independently of $t$, $j$ and $n$),
we can continue
\eqref{dkvntj},
\begin{equation*}
\le
\sum_{l=0}^k
\binom{k}{l} C n^l \beta^{\sqrt{n}} n^{2k-2l}
\left\|T_n\right\|_E
\le 
C n^{2k} \beta^{\sqrt{n}} \left\|T_n\right\|_E.
\end{equation*}
Collecting all the calculations in this paragraph,
for $t\in[a-\rho,a]$ we can write
\begin{equation}
\left|(T^*)^{(k)}(t) -T_n^{(k)}(t)\right|
\le
O\left( n^{2k} \beta^{\sqrt{n}}\right)
\left\|T_n\right\|_E.
\label{tnkcomparetnsk}
\end{equation}

Comparing the sup norms of $T_n$ and $T^*$,
we split the estimate into two cases
(see also \eqref{estimatesforVn}).
If $t\in E\setminus [a-2\rho_0,a]$, then
\begin{equation*}
\left|T^*(t)\right|
\le
\sum_{j=1}^N
\left|L_{\sqrt{n}}(t_j)\right|
\left|T_n(t_j)\right|
\le N C \beta^{\sqrt{n}} \left\|T_n\right\|_E
=o(1) \left\|T_n\right\|_E.
\end{equation*}
If $t\in[a-2\rho_0,a]$, then
\begin{multline*}
\left|T^*(t)\right|
\le 
\left|L_{\sqrt{n}}(t)\right|
\left|T_n(t)\right|
+
\sum_{j\ne j_0}
\left|L_{\sqrt{n}}(t_j)\right|
\left|T_n(t_j)\right|
\\ \le 
\left(1+N C \beta^{\sqrt{n}} \right)
\left\|T_n\right\|_E
=(1+o(1)) \left\|T_n\right\|_E.
\end{multline*}
These two estimates yield
\begin{equation}
\left\|T^*\right\|_E \le 
(1+o(1)) \left\|T_n\right\|_E.
\label{tstarsupnorm}
\end{equation}

Applying \eqref{tnkcomparetnsk}, \eqref{tstarsupnorm} and the previous case for $T^*$ (when $T^*$ is a polynomial of $U_N$),
we obtain \eqref{eq:Markov}
for T-sets and for arbitrary polynomials.

\medskip

Third step. 
Now let $E$ be an arbitrary set consisting of  finite number of intervals: $E=\cup_{j=1}^m [a_{2j-1},a_{2j}]$. 
Using the density of T-sets (see Section \ref{sec:density}),
there is a T-set $E'$ such that $E'\subset E$,
$a\in E'$ and
\begin{equation*}
\Omega(E_\bT, e^{ia})
\le 
\Omega(E'_\bT, e^{ia})
\le 
(1+\varepsilon) \Omega(E_\bT, e^{ia})
\end{equation*}
where $\varepsilon>0$ is arbitrary and $E'=E'(E,\varepsilon)$.
Here the first inequality comes from the monotonicity of $\Omega(.,.)$ 
(and from $E'\subset E$) and the second comes from the density result.
Obviously, $\left\|T_n\right\|_{E'}\le \left\|T_n\right\|_{E}$.
Now, applying the previous step (for arbitrary polynomials on T-sets), we can write
for $t\in[a-\rho,a]$
\begin{multline*}
\left|T_n^{(k)}(t)\right|
\le 
(1+o_{E'}(1))
\frac{8^k \pi^{2k}  }{(2k-1)!!}
n^{2k}
\Omega(E'_\bT,e^{ia})^{2k}
\left\|T_n\right\|_{E'}
\\ \le 
(1+o_{E}(1))
\frac{8^k \pi^{2k}  }{(2k-1)!!}
n^{2k}
\Omega(E_\bT,e^{ia})^{2k}
\left\|T_n\right\|_{E}
\end{multline*}
by letting $\varepsilon\rightarrow 0$ appropriately.

\medskip

Fourth step. 
Now let $E\subset (-\pi,\pi)$ be a compact set
which is regular (in the sense of Dirichlet problem). 
Obviously, the regularity of $E$ and $E_\bT$ are equivalent.

Consider the trigonometric polynomial $T_n Q_{n\varepsilon}$ of degree at most $n(1+\varepsilon)$ 
where $Q(.)=Q_{n\varepsilon}(.)$ is the fast decreasing polynomial with the following properties:
its degree is at most $n\varepsilon$,
$0\le Q(.)\le 1$, $Q(t)\le \exp(-\delta_1 n\varepsilon)$ 
for some $\delta_1>0$ on $t\in [-\pi,a-2\rho]\cup[a+2\rho,\pi]$,
$1-\exp(-\delta_1 n\varepsilon)\le Q(t)$ on $t\in[a-\rho,a+\rho]$ 
and $Q(a)=1$
(for existence, see Section \ref{sec:fastdecrtrig}).

Let $g_{E_{\mathbb{T}}}(\zeta,0)$ and  $g_{E_{\mathbb{T}}}(\zeta,\infty)$ be 
the Green functions of the domain $\overline{\mathbf{C}}\setminus E_{\mathbb{T}}$ with poles 
at the points $0$ and $\infty$, respectively. 
The regularity of the set $E$ (and $E_{\mathbb{T}}$ correspondingly) 
implies the continuity of $g_{E_{\mathbb{T}}}(\zeta,0)$ and $g_{E_{\mathbb{T}}}(\zeta,\infty)$ at all points different from $0$ and $\infty$, 
as well as the fact that these functions vanish at the points of $E_{\mathbb{T}}$. 
Therefore, for the $\delta_1>0$ 
there is a $d_1>0$, such that if $t\in \mathbf{R}$ and ${\rm dist}(t,E)\le d_1$, then 
\begin{equation}\label{eq:Greenest}
g_{E_{\mathbb{T}}}(e^{it},0)<\frac{\delta_1^2}{2}.
\end{equation}
We choose $m$ sufficiently so large that for the set $E_m^+$ the condition ${\rm dist}(t,E)\le d_1$ for all $t\in E_m^+$ is satisfied.

If $t\in E$ then 
\[
|T_n(t)Q_{n\varepsilon}(t)|\le \|T_n\|_E.
\]
If we write
\begin{multline*}
T_n(t)=\sum_{j=0}^{n} \left(A_j\cos jt+B_j\sin jt\right)
\\=
\sum_{j=0}^{n}\left({\rm Re} A_j \cos jt + 
{\rm Re} B_j \sin jt\right)+
i \sum_{j=0}^{n}\left({\rm Im} A_j \cos jt 
+ {\rm Im} B_j \sin jt\right),
\end{multline*}
we consider the algebraic polynomials 
\[
S^{(1)}_n(z)=
\sum_{j=0}^n\left({\rm Re} A_j-i{\rm Re} B_j\right)z^j, \ \ S^{(2)}_n(z)=\sum_{j=0}^n\left({\rm Im} A_j-i{\rm Im} B_j\right)z^j.
\]
It is easy to verify that $T_n(t)=F(e^{it})$ for all complex $t$, where
\begin{equation*}
F(z):=\frac12\left[S^{(1)}_n(z)+\overline{S^{(1)}_n\left(\frac{1}{\overline{z}}\right)}\right]
+
\frac{i}{2}\left[S^{(2)}_n(z)+\overline{S^{(2)}_n\left(\frac{1}{\overline{z}}\right)}\right]
\end{equation*}
is a rational function.
We note that $\|F\|_{E_{\mathbb{T}}}=\|T_n\|_{E}$ and apply an analog of the Bernstein-Walsh inequality  
(see e.g. \cite{Gonchar}, p. 64) 
to the rational function $F$ on $E_\bT$ and 
then use the fact that the domain $\overline{\mathbf{C}}\setminus E_{\mathbb{T}}$ is symmetric with respect to the unit circle. 
For simplicity, we put 
\begin{equation*}
g(z,w)=g_{\overline{\bC}\setminus E_\bT}(z,w)
\end{equation*} for 
Green's function of $E_\bT$.
So,  we have for $t\in\bR$ that
\begin{multline*}
|T_n(t)|=
\left|F\left(e^{it}\right)\right|
\le 
\|F\|_{E_{\mathbb{T}}} 
\exp\left(n\left(g(e^{it}, 0)
+
g(e^{it},\infty)\right) \right)
\\=
\|T_n\|_E 
\exp\left(2n g(e^{it},0)\right).
\end{multline*}
Now if $t\in E_m^+\setminus E$ then it follows from \eqref{fdtp:nonneg} 
and (\ref{eq:Greenest}) that
\begin{multline*}
|T_n(t)Q_{n\varepsilon}(t)|
\le 
\|T_n\|_{E} 
\exp\left(2 n g(e^{it},0)\right) \exp\left(-n\delta_1\right)
\\
\le 
\|T_n\|_E 
\exp\left(n\delta_1^2 -n\delta_1\right)
\le \|T_n\|_E
\end{multline*}
for  sufficiently large $n$, and hence $\|T_nQ_{n\varepsilon}\|_{E_m^+}\le \|T\|_E$.

For $t\in [a-\rho,a]$
\[
\left|(T_n Q_{n\varepsilon})^{(k)}(t)\right|
\ge 
\left|T^{(k)}(t)Q_{n\varepsilon}(t)\right|
-
\sum_{j=1}^{k}
\binom{k}{j}
\left|T_n^{(k-j)}(t)Q_{n\varepsilon}^{(j)}(t)\right|.
\]
Here $1-e^{-n\delta_1}\le Q_{n\varepsilon}(t)\le 1$ and 
by \eqref{roughMarkov}
\[
\|Q^{(j)}_{n\varepsilon}\|_E
\le 
C(n\varepsilon)^{2j}, \ \ \|T_n^{(j)}\|_E\le C n^{2j} \|T_n\|_E 
\]
with some constant $C$ for all $j=1,2,\dots, k$. 
Hence, if $t\in [a-\rho,a]$ we get from the previous 
step applied to the trigonometric polynomial $T_n(t)Q_{n\varepsilon}(t)$ 
on the set $E_m^+$ (which consists of finitely many intervals) 
that
\begin{multline*}
\left|T_n^{(k)}(t)\right|
\left(1-e^{-n\delta_1}\right)
\le 
\left|(T_nQ_{n\varepsilon})^{(k)}(t)\right|+\sum_{j=1}^{k}\binom{k}{j}C^2 \|T_n\|_{E} 
n^{2(k-j)}(n\varepsilon)^{2j}
\\
\le 
(1+o(1)) 8^k \pi^{2k} \frac{1}{(2k-1)!!}
\Omega\left((E^+_m)_\bT,e^{ia}\right)^{2k} 
\; 
(n(1+\varepsilon))^{2k}
\left\|T_n Q_{n\varepsilon}\right\|_{E_m^{+}} + C_1\varepsilon^2n^{2k}\|T_n\|_E \\
\le 
\frac{n^{2k}}{(2k-1)!!}\|T_n\|_E\left(\left(1+o(1)\right)(1+\varepsilon)^{2k}8^k\pi^{2k}
\Omega\left((E^+_m)_\bT,e^{ia}\right)^{2k}
+C_1\varepsilon^2\right).
\end{multline*}
Since $\varepsilon>0$ and $m$ are arbitrary, the inequality (\ref{eq:Markov}) follows from 
Lemma 
\ref{l:Omegalimit}.

Fifth step. 
The regularity condition can be removed using the sets $E_m^-$ 
and $(E_m^-)_{\mathbb{T}}$ from Ancona's theorem 
(interval condition \eqref{intervalcondition} implies $[a-\rho,a]\subset E$, hence $\capacity (E)>0$). 
Indeed, 
\begin{multline*}
\|T_n^{(k)}\|_{[a-\rho,a]}
\le 
\left(1+o_m(1)\right)\frac{n^{2k}}{(2k-1)!!}8^k\pi^{2k}
\Omega\left((E_m^-)_{\bT},e^{ia}\right)
\|T_n\|_{E_m^-}
\\
\le 
\left(1+o_m(1)\right)\frac{n^{2k}}{(2k-1)!!}8^k\pi^{2k}
\Omega\left((E_m^-)_{\bT},e^{ia}\right)
\|T_n\|_{E}
\end{multline*}
where $o_m(1)$ depends on $E_m^-$ too.

It follows from 
Lemma 
\ref{l:Omegalimit} that $\Omega(\left(E_m^-\right)_{\mathbb{T}},e^{ia})$ can  be made arbitrary close to $\Omega(E_{\mathbb{T}},e^{ia})$ by choosing $m$ large enough. Hence the inequality (\ref{eq:Markov}) holds in this case too. 

\bigskip

Now we investigate the 
sharpness, that is, we are going to establish \eqref{eq:Markovacc}.
As above, first we show it for the case when $E$ is a union of finitely many intervals. 
We select a T-set 
as in Section \ref{sec:density}
for which  $\Omega(E'_{\mathbb{T}},e^{ia})$ is close to $\Omega(E_{\mathbb{T}},e^{ia})$, say $\Omega(E'_{\mathbb{T}},e^{ia})\ge \Omega(E_{\mathbb{T}},e^{ia})(1-\varepsilon)$ for some given $\varepsilon>0$. 

By (\ref{eq:derofU})
\begin{equation}\label{eq:derofU1}
|U_N'(a)|=
8 \pi^2 N^2 \Omega(E'_\bT, e^{i a})^2.
\end{equation}
Now note that if $\mathcal{T}_l(x)=\cos(l\arccos (x))$ are classical Chebyshev polynomials, 
then $T_n(t):=\mathcal{T}_l(U_N(t))$ is a trigonometric polynomial of degree $lN$ for which 
\[
E'=\{x|\mathcal{T}_l(U_N(x))\in [-1,1]\}.
\]
Since 
\[
\left|\mathcal{T}_l^{k}(\pm 1)\right|= \frac{l^2(l^2-1)\dots(l^2-(k-1)^2)}{(2k-1)!!}=:C_{l,k}.
\]
and (\ref{eq:derofU1}) we get for $n=lN$ as before 
\[
\left|T_n^{(k)}(a)\right|
=
\left|(\mathcal{T}_l(U_N))^{(k)}(a)\right|
=
(1\pm o(1))C_{l,k}N^{2k}8^k \pi^{2k}  \Omega(E'_\bT,e^{i a})^{2k},
\]
and here, in view of (\ref{eq:limitbeh}),
\[
C_{l,k}N^{2k}\Omega(E'_\bT,e^{i a})^{2k}\ge (1-o(1))\frac{l^{2k}}{(2k-1)!!}N^{2k}\Omega(E_\bT,e^{i a})^{2k}(1-\varepsilon)^{2k}.
\]

Since $E\subset E'$ we have 
\[
\|T_n\|_E\le \|T_n\|_{E'}=\|\mathcal{T}_l\|_{[-1,1]}=1,
\]
and so from $n=lN$ we get 
\[
\left|T_n^{(k)}(a)\right|\ge (1-o(1))^2(1-\varepsilon)^{2k} \frac{n^{2k}}{(2k-1)!!}8^k \pi^{2k}  \Omega(E_\bT,e^{i a})^{2k}\|T_n\|_E.
\]
This is only for integers $n$ of the form $n=lN$. 
For others just use 
$T_n(t)=\mathcal{T}_{[n/N]}(U_N(t))+\delta \cos(nt)$ with $\delta>0$ very small. 
Since here $\varepsilon=\varepsilon_N>0$ is arbitrary, (\ref{eq:Markovacc}) follows if we let $N$ tend to $\infty$ slowly and at the same time $U_N^{-1}[-1,1]$ 
approaches 
$E$, as $n\rightarrow \infty$ (in which case we have $\varepsilon_N\rightarrow 0$).

In the general case we consider the sets $E_m^+$ that are unions of finitely many intervals. Hence, we may  use the last result for $E_m^+$, namely,
there is a sequence of nonzero trigonometric polynomials $\{T_{m,n}\}_{n=1}^{\infty}$, $\deg(T_{m,n})\le n$, such that 
\[
\left|T_{m,n}^{(k)}(a)\right|\ge (1-o_{E_m^+}(1)) n^{2k} 
\Omega\left((E_m^+)_\bT,e^{i a}\right)^{2k} 
\frac{8^k \pi^{2k}}{(2k-1)!!} 
\left\|T_{m,n}\right\|_{E_m^+},
\]
where $o_{E_m^+}(1)$ depends on $E_m^+$ and it tends to $0$ as $n\rightarrow \infty$ for any fixed $m$. Since $E\subset E_m^+$, we have $\|T_{m,n}\|_{E_m^+}\ge \|T_{m,n}\|_{E}$ and hence 
\[
\left|T_{m,n}^{(k)}(a)\right|\ge (1-o_{E_m^+}(1)) n^{2k} 
\Omega\left((E_m^+)_\bT,e^{i a}\right)^{2k} 
\frac{8^k \pi^{2k}}{(2k-1)!!} 
\left\|T_{m,n}\right\|_{E}.
\]
By 
Lemma 
\ref{l:Omegalimit} and choosing $m$ sufficiently large, 
$\Omega\left((E_m^+)_\bT,e^{i a}\right)$ 
can be made arbitrary close to $\Omega(E_\bT,e^{i a})$. Therefore, (\ref{eq:Markovacc}) follows for $T_{n}:=T_{m_n,n}$ if $m_n$ goes slowly to infinity as $n\rightarrow \infty$. 
\end{proof}

Now if 
$H$ denotes the shorter 
arc on $\mathbb{T}$ connecting the points $e^{i(a-\rho)}$ and $e^{ia}$ 
then we have the following assertion.
\begin{corollary}\label{cor:foralg}
Under the conditions mentioned above for any algebraic polynomial $P_n$ with degree $n$, we have
\begin{equation}\label{eq:MarkovAlg}
\left\|P_n^{(k)}\right\|_H
\le 
(1+o(1)) n^{2k} \Omega({E}_\bT,e^{i a})^{2k} \frac{2^k \pi^{2k}}{(2k-1)!!} 
\left\|P_n\right\|_{E_{\mathbb{T}}}.
\end{equation}

This inequality is sharp, for there is a sequence of polynomials $P_n \not\equiv 0$, $n=1,2,\dots$, such that
\begin{equation}\label{eq:MarkovAlgrev}
\left|P_n^{(k)}(e^{i a})\right|
\ge 
(1-o(1)) n^{2k} \Omega(E_\bT,e^{i a})^{2k} \frac{2^k \pi^{2k}}{(2k-1)!!} 
\left\|P_n\right\|_{E_{\mathbb{T}}}.
\end{equation}
The quantity $o(1)$ depends on $E$ and $k$ and tends to $0$ as $n\rightarrow\infty$.
\end{corollary}
\begin{proof}
We may assume that $n$ is even (because $(n+1)^2/n^2=1+o(1)$). We consider the trigonometric polynomial $T_{n/2}(t)=e^{-itn/2}P_n\left(e^{it}\right)$. 
So, (\ref{eq:MarkovAlg}) follows now from
applying Theorem~\ref{highertrigMarkov} to $T_{n/2}$.

Concerning (\ref{eq:MarkovAlgrev}), existence of such polynomials, in view of the remark above, follows from existence of trigonometric polynomials $T_{n}$ for which (\ref{eq:Markovacc}) holds.

\end{proof}

\section{Higher order Bernstein-type inequalities and their sharpness}

Let $E\subset (-\pi,\pi)$ be a compact subset,
and fix a point $z_0=e^{i t_0}$ which is
in the one dimensional interior of $E_\bT$.
That is, $\{\exp(it): t_0-\delta<t<t_0+\delta\}\subset E_\bT$ for some small $\delta>0$.
Denote by $\partial/\partial {\bf n}_+$ and $\partial/\partial {\bf n}_-$ 
the outward and inward normal derivatives (w.r.t. the unit circle) correspondingly.
Then  (see \cite{NTRiesz}, formulas (23) and (24) on p. 349) 
\begin{equation*}
\frac12 \left(1+2\pi\omega_{E_{\mathbb{T}}}\left(e^{it}\right)\right)
=
\frac{\partial g(e^{it},\infty)}{\partial {\bf n}_+}
=
\max \left(\frac{\partial g(e^{it},\infty)}{\partial {\bf n}_+}, \frac{\partial g(e^{it},\infty)}{\partial {\bf n}_{-}}\right)
\end{equation*}
where 
$g(z,w)=g_{\overline{\bC} \setminus E_\bT}(z,w)$ 
is Green's function of $\overline{\bC}\setminus E_\bT$
and $\omega_{E_\bT}(.)$ denotes the density 
of the equilibrium measure (w.r.t. arc length on the unit circle).

\medskip

Now let us consider higher order 
Bernstein-type inequalities for trigonometric polynomials.  

\begin{theorem}\label{thm:hBern}
Let $E\subset (-\pi,\pi)$ be 
a compact set 
and $k$ be a positive integer.
Fix a closed 
interval
$E_0\subset \interior E$
(subset of the one dimensional interior of $E$).
Then
there exists $C=C(E,E_0,k)>0$ such that
for all trigonometric polynomial $T_n$ with degree $n$,
we have for $t\in E_0$
\begin{equation}
\label{eq:Bernstein}
\left|T_n^{(k)}(t)\right|
\le 
(1+o(1)) n^{k} \left(2\pi\omega_{E_{\mathbb{T}}}\left(e^{it}\right)\right)^k
\left\|T_n\right\|_E .
\end{equation}
where $o(1)$ is uniform in $t\in E_0$ 
and  uniform among all trigonometric polynomials having degree at most $n$ and
tends to $0$ as $n\rightarrow \infty$.
\end{theorem}

\begin{proof}

We prove the theorem by induction on $k$, the case $k=1$ was done in \cite[Theorem 4]{MR2069196}.

Let
$$
V(t)=2\pi\omega_{E_{\mathbb{T}}}\left(e^{it}\right). 
$$
Select a closed set $E_0^*\supset E_0$ such that $E_0^*$ has no common endpoints either with $E_0$ or with $E$. 

Consider any $\delta>0$ such that the intersection of $E$ with the $\delta$-neighborhood of $E_0$ is 
still subset of 
of $E_0^*$, 
and set $f_{k,n,t_0}(t):=T_n^{(k)}(t)Q(t)$, where $Q(t)=Q_{n^{1/3}}(t)$  is a fast decreasing trigonometric polynomial from Theorem~\ref{thm:fdtp} for  $t_0\in E_0$ ($\alpha'$ and $\beta'$ from
Theorem~\ref{thm:fdtp} are chosen such a way that 
the interval $[\alpha',\beta']$ is in the $\delta$-neighborhood of $E_0$).

By (\ref{fdtp:low}) and (\ref{roughMarkov}), for this $f_{k,n,t_0}$ we have the upper bound 
\begin{equation*}
O(n^{2k})\exp\left(-\delta_1n^{1/3}\right)\|T_n\|_{E}
=
o(1)\|T_n\|_{E}
\end{equation*}
on $E$ outside the $\delta$-neighborhood of $t_0$ with 
$\delta_1>0$ (uniform in $t_0\in E_0$).  

In the $\delta$-neighborhood of any $t_0\in E_0$,  by $\|Q\|_{E}\le 1$ and by induction hypothesis applied to $T_n$ and to $E_0^*$, we have
\begin{multline*}
|f_{k,n,t_0}(t)|
\le 
(1+o(1))n^{k} \|T_n\|_{E}V(t)^k
\le 
(1+o(1))n^k(1+\varepsilon)^k
\|T_n\|_E V(t_0)^k,
\end{multline*}
where $\varepsilon\rightarrow 0$ as $\delta\rightarrow 0$. 
Here we used that by the continuity of $V(t)$, if 
$t_0\in E_0$ 
and 
$|t-t_0|<\delta$, then $V(t)\le (1+\varepsilon)V(t_0)$ 
with some $\varepsilon$ that tends to $0$ as $\delta\rightarrow 0$. 
Therefore, $f_{k,n,t_0}(t)$ is a trigonometric polynomial in $t$ of degree at most $n+n^{1/3}$ for which 
\begin{equation*}
\|f_{k,n,t_0}\|
\le 
(1+o(1))n^k \|T_n\|_E V(t_0)^k.
\end{equation*}
Upon applying Lukashov's theorem from \cite[Theorem 4]{MR2069196} to the trigonometric polynomial $f_{k,n,t_0}(t)$ we obtain
\begin{equation}\label{eq:derforInd}
|f'_{k,n,t_0}(t_0)|
\le (1+o(1))n^{k+1}
\|T_n\|_E V(t_0)^{k+1}.
\end{equation}
Since (recall that $Q(t_0)=1$)
\begin{equation*}
f'_{k,n,t_0}(t_0)
=
T_n^{(k+1)}(t_0)
+T_n^{(k)}(t_0)(Q(t_0))',
\end{equation*}
and the second term on the right is 
at most 
$O(n^k)O(n^{2/3})\|T_n\|_E$ 
in modulus,
by \eqref{roughMarkov} and by the induction assumption, 
from \eqref{eq:derforInd} we get \eqref{eq:Bernstein}. 
It follows from the proof that the estimate is uniform in $t_0\in E_0$.

\end{proof}

\begin{corollary}\label{cor:hBernAlg} 
Let $E\subset (-\pi,\pi)$ be again a 
compact set 
and 
$k$ be a positive integer.
Fix a closed 
interval 
$E_0\subset \interior E$.
Then
there exists $C=C(E,E_0,k)>0$ such that
for all algebraic polynomial $P_n$ with degree $n$,
we have for 
$z=e^{it}$, $t\in E_0$ 
\begin{equation}
\label{roughBernsteinAlg}
\left|P_n^{(k)}(z)\right|
\le 
(1+o(1)) \frac{n^{k}}{2^k} 
\left(1+2\pi\omega_{E_{\bT}}(z)\right)^k
\left\|P_n\right\|_{E_\bT}  
\end{equation}
where $o(1)$ is
uniform in $z=e^{it}$, $t\in E_0$
and  independent of $P_n$, 
but it tends to $0$
as $n\rightarrow \infty$.
\end{corollary}

\begin{proof}
As in the proof of 
Corollary 
\ref{cor:foralg}, we may assume that $n$ is even (because $(n+1)^2/n^2=1+o(1)$) and consider the trigonometric polynomial $T_{n/2}(t)=e^{-itn/2}P_n\left(e^{it}\right)$. 
By 
Theorem 
\ref{thm:hBern}, we get 
\begin{multline*}
(1+o(1))\frac{n^k}{2^k} 
\left(2\pi\omega_{E_{\mathbb{T}}}\left(e^{it}\right)\right)^k 
\|T_{n/2}\|_E
\ge 
|T^{(k)}_n\left(t\right)|
\\
\ge 
\left|\left(P_n\left(e^{it}\right)\right)^{(k)}\right|
-\left|\sum_{j=0}^{k-1}\binom{k}{j}\left(P\left(e^{it}\right)\right)^{(j)}
\left(e^{-itn/2}\right)^{(k-j)}\right|.
\end{multline*}
It, together with  Fa\`a di Bruno's formula 
\eqref{faadibruno} 
and
Theorem 
\ref{thm:hBern} yields that
\begin{multline*}
\left|P_n^{(k)}(z)\right|\le 
(1+o(1))
\frac{n^k}{2^k}\left( \left(2\pi\omega_{E_{\mathbb{T}}}(z)\right)^{k}+\sum_{j=0}^{k-1}\binom{k}{j} 
\left(2\pi\omega_{E_{\bT}}\left(z\right)\right)^{j} \right)
\|P_n\|_{E_\bT} 
\\
\le (1+o(1)) \frac{n^k}{2^k}(1+2\pi
\omega_{E_\bT}\left(z\right))^k 
\|P_n\|_{E_\bT}. 
\end{multline*}
\end{proof}

Corollary~\ref{cor:hBernAlg} extends Theorem 1 of the paper \cite{NTRiesz} 
to higher derivatives of algebraic polynomials and 
the 
proof of sharpness is similar to the proof of 
\cite{NTRiesz}, Theorem 2. 

\begin{theorem}
Under assumption of 
Corollary 
~\ref{cor:hBernAlg}, 
inequality (\ref{roughBernsteinAlg}) is sharp, 
that is, 
there is a sequence of polynomials $P_n \not\equiv 0$, $n=1,2,\dots$, such that
\begin{equation*}
\left|P_n^{(k)}(z)\right|
\ge 
(1-o(1)) \frac{n^{k}}{2^k} \left(1+2\pi\omega_{E_{\mathbb{T}}}\left(z\right)\right)^k
\left\|P_n\right\|_{E_{\mathbb{T}}} .
\end{equation*}
The quantity $o(1)$ depends on 
$E$ 
and $k$ and tends to $0$ as $n\rightarrow\infty$.
\end{theorem}

\begin{proof}
We enclose 
$E_\bT$ 
into a set $G$ with the following properties: 
\begin{itemize}
\item $G$ is a finite union of disjoint $C^2$ smooth Jordan domains: 
there are finitely many  disjoint $C^2$ Jordan curves $S_1,\dots, S_m$ such that if $G_j$ is the bounded connected components of $\overline{\mathbf{C}}\setminus S_j$, 
then $\overline{G}=\cup_{j=1}^m\overline{G}_j$,
\item 
$E_{\mathbb{T}}$ 
is a boundary arc of the boundary $\partial G$,
\item the component of $G$ that contains $z$ lies in the closed unit disk,
\item every point of $G$ is of distance $\le \eta$ from a point of 
$E_{\mathbb{T}}$, 
where $\eta$ is a given positive number.
\end{itemize}

Then the boundary $\Gamma=\partial G= \cup_{j=1}^m S_j$ is a family of disjoint Jordan curves. 
Furthermore, let $\mathbf{n}_+=z$ be the normal at $z$ to $\Gamma$ pointed to the interior of $\Omega=\overline{\mathbf{C}}\setminus G$. 

If  $\varepsilon>0$ is given, then for sufficiently small $\eta$ we have  
(see e.g.~\cite{NTHilbert}, pp.~350-351 
\begin{equation}\label{eq:normdercomp}
\frac{\partial g_{\Omega}(z,\infty)}{\partial \mathbf{n}_+} \ge (1-\varepsilon) \frac{\partial g_{\overline{\mathbf{C}}\setminus E_{\mathbb{T}}}(z,\infty)}{\partial \mathbf{n}_+}.
\end{equation}
 By the sharp form of the Hilbert lemniscate theorem 
\cite{NTHilbert}, Theorem 1.2, 
there is a Jordan curve $\sigma$ such that 
 \begin{itemize}
 \item $\sigma$ contains 
$\Gamma$ in its interior except for the point $z$, where the two curves touch each other,
\item $\sigma$ is a lemniscate, i.e. $\sigma=\{\zeta:\  |V_N(\zeta)|=1\}$ 
for some
algebraic polynomial $V_N$ of degree $N$, and 
\item \begin{equation}\label{eq:normdercomp1}
\frac{\partial g_{\overline{\bC}\setminus \sigma}(z,\infty)}{\partial \mathbf{n}_+}
\ge 
(1-\varepsilon)\frac{\partial g_{\Omega}(z,\infty)}{\partial \mathbf{n}_+}.
\end{equation}
 \end{itemize}
 
We may assume that  $V_N'(z)>0$. 
The Green's function of the outer domain of $\sigma$ is 
$\frac{1}{N}\log|V_N(.)|$, 
and its normal derivative is
\begin{equation*}
\frac{\partial g_{\overline{\bC}\setminus \sigma}(z,\infty)}{\partial \mathbf{n}_+} 
= \frac{1}{N} |V'_N(z)|
=\frac{1}{N} V'_N(z).
\end{equation*}

Consider now, for all large $n$, the polynomials $P_n(.)=V_N(.)^{[n/N]}$. 
This is a polynomial of degree at most $n$, 
its supremum norm on $\sigma$ is $1$, and 
by Fa\`a di Bruno formula \eqref{faadibruno},
it can be shown that 
(see also \cite{BJP}, subsection 10.2)  
\begin{equation*}
\left|P_n^{(k)}(z)\right|
=
n^k
\left(\frac{\partial g_{\overline{\bC}\setminus \sigma}(z,\infty)}{\partial \mathbf{n}_+}\right)^k
+O(n^{k-1}).
\end{equation*}
Thus, in view of (\ref{eq:normdercomp}) and (\ref{eq:normdercomp1}), we may continue 
\begin{equation*}
\left|P_n^{(k)}(z)\right|
\ge 
(1-\varepsilon)^{2k}n^k\left(\frac{\partial 
g_{\overline{\bC}\setminus E_{\bT}}
(z,\infty)}{\partial \mathbf{n}_+}\right)^k
+
O(n^{k-1}).
\end{equation*}
Note also that 
$\|P_n\|_{E_{\mathbb{T}}}\le \|P_n\|_{\sigma}=1$ 
by the maximum principle.
 \end{proof}

\begin{corollary}
Under assumption of 
Theorem 
\ref{thm:hBern}, inequality (\ref{eq:Bernstein}) is sharp, for there is a sequence of trigonometric polynomials $T_n \not\equiv 0$, $n=1,2,\dots$, such that
\begin{equation*}
\left|T_n^{(k)}(t)\right|
\ge 
(1-o(1)) n^{k} \left(2\pi\omega_{E_{\mathbb{T}}}\left(e^{it}\right)\right)^k
\left\|T_n\right\|_E. 
\end{equation*}
where $o(1)$ depends on $E$ and $k$ and tends to $0$ as $n\rightarrow\infty$.
\end{corollary}

\begin{proof}
Existence 
of such trigonometric polynomials $T_{n}$ follows 
immediately from the existence 
of corresponding (in the sense of the proof of Corollary~\ref{cor:hBernAlg}) algebraic polynomials $P_{2n}$ from 
Corollary 
\ref{cor:hBernAlg}.
\end{proof}

\section*{Acknowledgement} 
The research of the first author has been supported by Russian Science Foundation 
under project 14-11-00022
(in the part concerning polynomial inequalities).

The second author was supported by the János Bolyai 
Scholarship of Hungarian Academy of Sciences.

\providecommand{\bysame}{\leavevmode\hbox to3em{\hrulefill}\thinspace}
\providecommand{\MR}{\relax\ifhmode\unskip\space\fi MR }
\providecommand{\MRhref}[2]{%
  \href{http://www.ams.org/mathscinet-getitem?mr=#1}{#2}
}
\providecommand{\href}[2]{#2}

Sergei Kalmykov
\\
School of Mathematical Sciences, Shanghai Jiao Tong University,
800 Dongchuan RD, Shanghai, 200240, P.R. China 
\\
Far Eastern Federal University, 8 Sukhanova Street, Vladivostok,
690950, Russia
\\
email address: \href{mailto:sergeykalmykov@inbox.ru}{sergeykalmykov@inbox.ru}

\medskip{}

Béla Nagy
\\
MTA-SZTE Analysis and Stochastics Research Group, Bolyai Institute,
University of Szeged, Szeged, Aradi v. tere 1, 6720, Hungary
\\
email address: \href{mailto:nbela@math.u-szeged.hu}{nbela@math.u-szeged.hu}

\end{document}